\documentclass[a4paper,11pt]{article}
\usepackage[T1]{fontenc}
\usepackage[english]{babel}  
\usepackage{amsthm,amssymb,amsfonts,amsmath,color,graphicx,float,caption,hyperref}
\usepackage[authoryear,round,sort]{natbib}
\usepackage{geometry}
\usepackage[noend]{algorithmic}  
\usepackage{algorithm}    
\usepackage{listings}

\newtheorem{lemma}{Lemma}
\newtheorem{Prop}{Proposition}
 
\newtheorem{theorem}{Theorem}
\newtheorem{definition}{Definition}

\newcommand{\R}{\mathbb{R}}

\newcommand{\indic}[1]{ \mathbf{1}_{#1}}

\newcommand{\N}{\mathbb{N}}
\newcommand{\M}{\mathbb{M}}
\newcommand{\E}{\mathbb{E}}

\newcommand{\V}{\mathbb{V}}
\newcommand{\W}{\mathbb{W}}
\newcommand{\Z}{\mathbb{Z}}
\newcommand{\eps}{\varepsilon}

\DeclareMathOperator{\reach}{reach}
\DeclareMathOperator{\wfs}{wfs}

\newcommand{\tda}{{\sc tda}}

  \newcommand{\bottle}{\operatorname{d_b}}
  \newcommand{\dgm}{\operatorname{dgm}}
\newcommand{\Filt}{\operatorname{\mathrm{Filt}}}
\newcommand{\X}{\mathbb{X}}  
\newcommand{\Y}{\mathbb{Y}}  
\newcommand{\hX}{\widehat{\mathbb{X}}}
\newcommand{\dhaus}{\operatorname{Haus}}

\newcommand{\Rips}{\operatorname{\mathrm{Rips}}}
\newcommand{\Cech}{\operatorname{\mathrm{Cech}}}

\newcommand{\lscape}{\lambda} 
\newcommand{\Xm}{X}
\newcommand{\Ym}{Y}

\graphicspath{{"./Fig/"},{"./"}}

\definecolor{codegreen}{rgb}{0,0.6,0}
\definecolor{codegray}{rgb}{0.5,0.5,0.5}
\definecolor{codepurple}{rgb}{0.58,0,0.82}
\definecolor{backcolour}{rgb}{0.95,0.95,0.92}
 
\lstdefinestyle{mystyle}{
    backgroundcolor=\color{backcolour},   
    commentstyle=\color{codegreen},
    keywordstyle=\color{magenta},
    numberstyle=\tiny\color{codegray},
    stringstyle=\color{codepurple},
    basicstyle=\footnotesize,
    breakatwhitespace=false,         
    breaklines=true,                 
    captionpos=b,                    
    keepspaces=true,                 
    numbers=left,                    
    numbersep=5pt,                  
    showspaces=false,                
    showstringspaces=false,
    showtabs=false,                  
    tabsize=2
}

\title{An introduction to Topological Data Analysis: fundamental and practical aspects for data scientists}

\author{Fr\'ed\'eric Chazal and Bertrand Michel}

\date{\today}

\begin{document}

\maketitle

\begin{abstract}
Topological Data Analysis (\tda ) is a recent and fast growing field providing a set of new topological and geometric tools to infer relevant features for possibly complex data. This paper is a brief introduction, through a few selected topics, to basic fundamental and practical aspects of \tda\ for non experts. 
\end{abstract}

%\fred{Faire une passe complete sur les ref avec mise-a-jour}

\section{Introduction and motivation}

%\fred{**** Citer Barannikov !!!!!}

%\fred{Faire une passe sur l'intro pour la rafraichir $\to$ Fred et Bertrand}

Topological Data Analysis (\tda) is a recent field that emerged from various works in applied (algebraic) topology and computational geometry during the first decade of the century. Although one can trace back geometric approaches for data analysis quite far in the past, \tda\ really started as a field with the pioneering works of \cite{elz-tps-02} and \cite{zc-cph-05} in persistent homology and was popularized in a landmark paper in 2009 \cite{c-td-09}. 
\tda\ is mainly motivated by the idea that topology and geometry provide a powerful approach to infer robust qualitative, and sometimes quantitative, information about the structure of data - see, e.g. \cite{c-htda-16}.

\tda~aims at providing well-founded mathematical, statistical and algorithmic methods to infer, analyze and exploit the complex topological and geometric structures underlying data that are often represented as point clouds in Euclidean or more general metric spaces. During the last few years, a considerable effort has been made to provide robust and efficient data structures and algorithms for \tda\ that are now implemented and available and easy to use through standard libraries such as the Gudhi library \footnote{\url{https://gudhi.inria.fr/}} (C++ and Python) \cite{maria2014gudhi} and its R software interface \cite{fasy2014introduction}, 
Dionysus\footnote{\url{http://www.mrzv.org/software/dionysus/}}, PHAT\footnote{\url{https://bitbucket.org/phat-code/phat}}, DIPHA\footnote{\url{https://github.com/DIPHA/dipha}}, or Giotto \footnote{\url{https://giotto-ai.github.io/gtda-docs/0.4.0/library.html}}. Although it is still rapidly evolving, \tda\ now provides a set of mature and efficient tools that can be used in combination or complementary to other data sciences tools. 

%Its overall goal is two-fold. 
%First, \tda\ aims at providing mathematical results and methods to infer, analyze and exploit the complex topological and geometric structures underlying data
%that are often represented as point clouds in Euclidean or more general metric spaces. Second, it intends to give access to robust and efficient data structures
%and algorithms to represent these data and that are amenable to precise analysis. 

%\begin{figure}[h!] 
%\centering \label{fig:TDApipeline}
%\includegraphics[width = \columnwidth]{TDApipeline}
%\caption{The \tda\ pipeline.}
%\end{figure}

\paragraph{The \tda pipeline.}
\tda\ has recently known developments in various directions and application fields. There now exist a large variety of methods inspired by topological and geometric approaches. Providing a complete overview of all these existing approaches is beyond the scope of this introductory survey. However, many standard ones rely on the following basic pipeline that will serve as the backbone of this paper:
%\footnote{This pipeline does not pretend to cover the whole of \tda\ or all of the research conducted in the \DataShape\ project, but to give a rough intuitions on the \tda\ approach and methodology}:
\begin{enumerate}
\item The input is assumed to be a finite set of points coming with a notion of distance - or similarity - between them. This distance can be induced by the metric in the ambient space (e.g. the Euclidean metric when the data are embedded in $\R^d$) or come as an intrinsic metric defined by a pairwise distance matrix. The definition of the metric on the data is usually given as an input or guided by the application. It is however important to notice that the choice of the metric may be critical to reveal interesting topological and geometric features of the data.\\
%When probabilistic or statistical aspects are considered, some assumptions are usually made on the generative model from which the data have been sampled. 
%However, some geometric and topological considerations may sometimes help to define a metric that highlights interesting features. 

\item A ``continuous'' shape is built on top of the data in order to highlight the underlying topology or geometry. This is often a simplicial complex or a nested family of simplicial complexes, called a filtration, that reflects the structure of the data at different scales. Simplicial complexes can be seen as higher dimensional generalizations of neighboring graphs that are classically built on top of data in many standard data analysis or learning algorithms. The challenge here is to define such structures that are proven to reflect relevant information about the structure of data and that can be effectively constructed and manipulated in practice. 
%This may result in a full reconstruction, typically a triangulation, of the shape underlying the data from which topological/geometric invariants can be easily extracted or in rougher approximations from which the extraction of relevant information requires specific methods.  

\item Topological or geometric information is extracted from the structures built on top of the data. 
This may either results in a full reconstruction, typically a triangulation, of the shape underlying the data from which topological/geometric features can be easily extracted or, in crude summaries or approximations from which the extraction of relevant information requires specific methods, such as e.g. persistent homology.  
%Persistent homology is the most used tool to extract such features. 
Beyond the identification of interesting topological/geometric information and its visualization and interpretation, the challenge at this step is to show its relevance, in particular its stability with respect to perturbations or presence of noise in the input data. For that purpose, understanding the statistical behavior of the inferred features is also an important question.

\item The extracted topological and geometric information provides new families of features and descriptors of the data. They can be used to better understand the data - in particular through visualization- or they can be combined with other kinds of features for further analysis and machine learning tasks. These information can also be used to design well-suited data analysis and machine learning models. Showing the added-value and the complementarity (with respect to other features) of the information provided by \tda\ tools is an important question at this step. 
\end{enumerate}

\paragraph{\tda\ and statistics.}
Until quite recently,  the theoretical aspects of TDA and topological inference mostly relied on deterministic approaches. These deterministic approaches do not take into account the random nature of data and the intrinsic variability of the topological quantity they infer.  Consequently, most of the corresponding methods remain exploratory, without being able to efficiently distinguish between  information and what is sometimes called the "topological noise".

A statistical approach to TDA means that we consider data as generated from an unknown distribution, but also that the inferred topological features by TDA methods  are seen as estimators of topological quantities describing an underlying  object. Under this approach, the unknown object usually corresponds to the support of the data distribution (or part of it).
%(or at least is close to this support). However, this support does not always have a physical existence; for instance, galaxies in the universe are organized along filaments but these filaments do not physically exist. 
The main goals of a statistical approach to topological data analysis can be summarized as the following list of problems:
\begin{description}
\item [Topic 1:] proving consistency and studying the convergence rates of TDA methods.
%\item providing a description of difficulty of the topological inference problem in term of convergence rates ;
\item[Topic 2:] providing confidence regions for topological features and  discussing the significance of the  estimated topological quantities.
\item[Topic 3:] selecting relevant scales at which the topological phenomenon should be considered, as a function of observed data.
\item[Topic 4:]  dealing with outliers and providing robust methods for TDA. 
\end{description}

\paragraph{Applications of \tda\ in data science.}
%\fred{Il faut qu'on mette des ref plus recentes et qu'on adapte le discours}
On the application side, many recent promising and successful results have demonstrated the interest of topological and geometric approaches in an increasing number of fields such has, e.g., 
material science \cite{kramar2013persistence,0957-4484-26-30-304001}
 3D shape analysis~\cite{socg-pbsds-10,turner2014persistent},
 image analysis~\cite{rieck2020uncovering, qaiser2019fast},
 multivariate time series analysis~\cite{seversky2016time, umeda2017time, khasawneh2016chatter},
medicine~\cite{dindin2020topological},
 biology~\cite{yao2009topological} ,
 genomic~\cite{carriere2020topological}
 chemistry~\cite{lee2017quantifying}
or sensor networks~\cite{de2007homological}
 to name a few. It is beyond the scope to give an exhaustive list of applications of \tda.
On another hand, most of the successes of \tda\ result from its combination with other analysis or learning techniques - see Section \ref{sec:PHandLearning} for a discussion and references. So, clarifying the position and complementarity of \tda\ with respect to other approaches and tools in data science is also an important question and an active research domain.

\medskip
The overall objective of this survey paper is two-fold. First, it intends to provide data scientists with a brief and comprehensive introduction to the mathematical and statistical foundations of \tda. For that purpose, the focus is put on a few selected, but fundamental, tools and topics: simplicial complexes (Section \ref{sec:simplicial-complexes}) and their use for exploratory topological data analysis (Section \ref{sec:mapper}), geometric inference (Section \ref{sec:geometric-inference}) and persistent homology theory (Section \ref{sec:persistent-homology}) that play a central role in \tda. 
Second, this paper also aims at demonstrating how, thanks to the recent progress of software, \tda\ tools can be easily applied in data science. In particular, we show how the Python version of the Gudhi library allows to easily implement and use the \tda\ tools presented in this paper (Section \ref{sec:gudhi}). Our goal is to quickly provide the data scientist with a few basic keys - and relevant references - to get a clear understanding of the basics of \tda\ to be able to start to use \tda\ methods and software for his own problems and data.

\section{Metric spaces, covers and simplicial complexes} \label{sec:simplicial-complexes}
%\fred{
%+ covers <-> (overlapping) clusters \\
%+ nerve and simplicial complexes : global/combinatorial structure of covers/clusterings
%+ Exploratory data analysis\\
%+ The example of Mapper\\
%}
%\fred{should we recall the definition of metric space and give the notation for balls?}

%\fred{TODO: relire attentivement} 

As topological and geometric features are usually associated to continuous spaces, data represented as finite sets of observations, do not directly reveal any topological information per se. A natural way to highlight some topological structure out of data is to ``connect'' data points that are close to each other in order to exhibit a global continuous shape underlying the data. Quantifying the notion of closeness between data points is usually done using a distance (or a dissimilarity measure), and it often turns out to be convenient in \tda\ to consider data sets as discrete metric spaces or as samples of metric spaces. 

\paragraph{Metric spaces.} Recall that a metric space $(M,\rho)$ is a set $M$ with a function $\rho: M \times M \to \R_+$, called a distance,  such that for any $x,y,z \in M$:\par
{\em i)} $\rho(x,y) \geq 0$ and $\rho(x,y)=0$ if and only if $x=y$,\par
{\em ii)} $\rho(x,y) = \rho(y,x)$ and,\par
{\em iii)} $\rho(x,z) \leq \rho(x,y) + \rho(y,z)$.
\par \noindent
Given a a metric space $(M,\rho)$, the set $\mathcal{K}(M)$ of its compact subsets can be endowed with the so-called \emph{Hausdorff distance}:
given two compact subsets $A, B \subseteq M$ the Hausdorff distance $d_H(A,B)$ between $A$ and $B$ is defined as the smallest non negative number $\delta$ such that for any $a \in A$ there exists $b \in B$ such that $\rho(a,b) \leq \delta$ and for any $b \in B$, there exists $a \in A$ such that $\rho(a,b) \leq \delta$ - see Figure \ref{fig:HausdorffDist}.
In other words, if for any compact subset $C \subseteq M$, we denote by $d(.,C) : M \to \R_+$ the distance function to $C$ defined by $d(x,C) := \inf_{c \in C} \rho(x,c)$ for any $x \in M$, then one can prove that the Hausdorff distance between $A$ and $B$ is defined by any of the two following equalities:
\begin{eqnarray*}
d_H(A,B) & = & \max \{ \sup_{b \in B} d(b,A) , \sup_{a \in A} d(a,B) \} \\
& = & \sup_{x \in M} | d(x,A) - d(x,B) | =  \| d(.,A) - d(.,B) \|_\infty
\end{eqnarray*}

\begin{figure}[h]
	\centering
		\includegraphics[width = 0.8 \columnwidth]{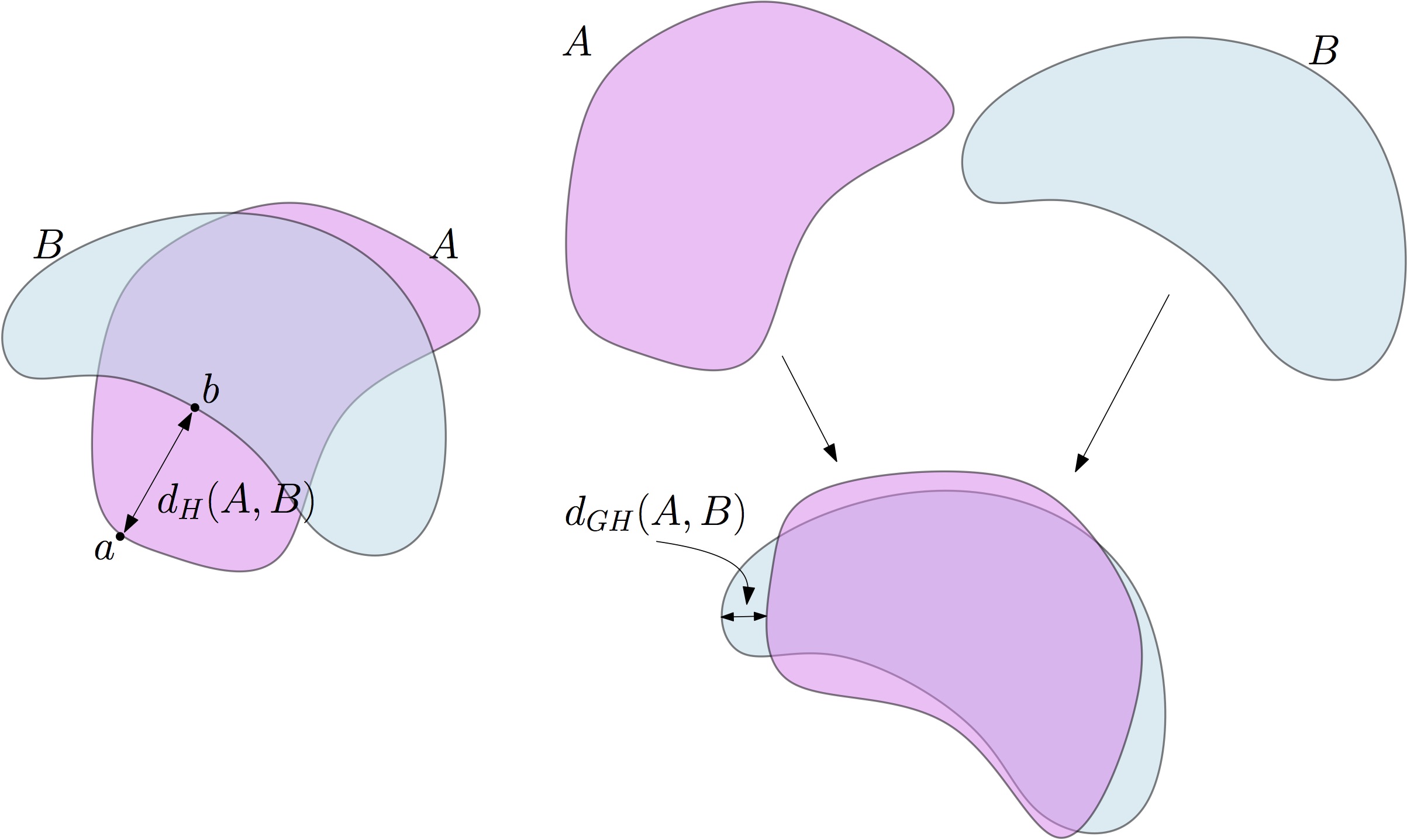} 
		\caption{Right: the Hausdorff distance between two subsets $A$ and $B$ of the plane. In this example, $d_H(A,B)$ is the distance between the point $a$ in $A$ which is the farthest from $B$ and its nearest neighbor $b$ on $B$. Left: The Gromov-Hausdorff distance between $A$ and $B$. $A$ can been rotated - this is an isometric embedding of $A$ in the plane - to reduce its Hausdorff distance to $B$. As a consequence, $d_{GH}(A,B) \leq d_H(A,B)$.} 
	\label{fig:HausdorffDist}
\end{figure}

It is a basic and classical result that the  Hausdorff distance is indeed a distance on the set of compact subsets of a metric space. From a \tda\ perspective it provides a convenient way to quantify the proximity between different data sets issued from the same ambient metric space. 
However, it sometimes occurs in that one has to compare data set that are not sampled from the same ambient space. Fortunately, the notion of Hausdorff distance can be generalized to the comparison of any pair of compact metric spaces, giving rise to the notion of \emph{Gromov-Hausdorff distance}. 

Two compact metric spaces $(M_1,\rho_1)$ and $(M_2,\rho_2)$ are \emph{isometric} if there exists a bijection $\phi : M_1 \to M_2$ that preserves distances, i.e. $\rho_2(\phi(x),\phi(y)) = \rho_1(x,y)$ for any $x,y \in M_1$. The Gromov-Hausdorff distance measures how far two metric space  are from being isometric. 

\begin{definition}
The Gromov-Haudorff distance $d_{GH}(M_1,M_2)$ between two compact metric spaces is the infimum of the real numbers $r \geq 0$ such that there exists a metric space $(M, \rho)$ and two compact subspaces $C_1, C_2 \subset M$ that are isometric to $M_1$ and $M_2$ and such that $d_H(C_1,C_2) \leq r$. 
\end{definition}

The Gromov-Hausdorff distance will be used later, in Section \ref{sec:persistent-homology}, for the study of stability properties persistence diagrams. 

\medskip
Connecting pairs of nearby data points by edges leads to the standard notion of neighboring graph from which the connectivity of the data can be analyzed, e.g. using some clustering algorithms. 
To go beyond connectivity, a central idea in TDA is to build higher dimensional equivalent of neighboring graphs by not only connecting pairs but also $(k+1)$-uple of nearby data points. The resulting objects, called simplicial complexes, allow to identify new topological features such as cycles, voids and their higher dimensional counterpart.

\paragraph{Geometric and abstract simplicial complexes.}
Simplicial complexes  can be seen as higher dimensional generalization of graphs. They are mathematical objects that are both topological and combinatorial, a property making them particularly useful for \tda.

Given a set $\X = \{ x_0,\cdots, x_k \} \subset \R^d$ of $k+1$ affinely independent points, the {\em $k$-dimensional simplex} $\sigma = [x_0, \cdots x_k ]$ spanned 
by $\X$ is the convex hull of $\X$. The points of $\X$ are called the {\em vertices} of $\sigma$ and the simplices spanned by the subsets of $\X$ are called the {\em faces} of $\sigma$. A {\em geometric simplicial complex} $K$ in $\R^d$ is a collection of simplices such that:\\
$i)$ any face of a simplex of $K$ is a simplex of $K$,\\
$ii)$ the intersection of any two simplices of $K$ is either empty or a common face of both.

The union of the simplices of $K$ is a subset of $\R^d$ called the underlying space of $K$ that inherits from the topology of $\R^d$. So, $K$ can also be seen as a topological space through its underlying space. Notice that once its vertices are known, $K$ is fully characterized by the combinatorial description of a collection of simplices satisfying some incidence rules. 
%This suggests a purely combinatorial notion of simplicial complexes that is independent of any embedding in an Euclidean space. 

Given a set $V$, an {\em abstract simplicial complex} with vertex set $V$ is a set $\tilde K$ of finite subsets of $V$ such that the elements of $V$ belongs to $\tilde K$ and for any $\sigma \in \tilde K$ any subset of $\sigma$ belongs to $\tilde K$. The elements of $\tilde K$ are called the faces or the simplices of $\tilde K$. The dimension of an abstract simplex is just its cardinality minus $1$ and the dimension of $\tilde K$ is the largest dimension of its simplices. Notice that simplicial complexes of dimension $1$ are graphs. 

The combinatorial description of any geometric simplicial $K$ obviously gives rise to an abstract simplicial complex $\tilde K$. The converse is also true: one can always associate to an abstract simplicial complex $\tilde K$, a topological space $|\tilde K|$ such that if $K$ is a geometric complex whose combinatorial description is the same as $\tilde K$, then the underlying space of $K$ is homeomorphic to $|\tilde K|$. Such a $K$ is called a {\em geometric realization} of $\tilde K$. 
As a consequence, abstract simplicial complexes can be seen as topological spaces and geometric complexes can be seen as geometric realizations of their underlying combinatorial structure.  So, one can consider simplicial complexes at the same time as combinatorial objects that are well-suited for effective computations and as topological spaces from which topological properties can be inferred. 

\paragraph{Building simplicial complexes from data.}
Given a data set, or more generally a topological or metric space, there exist many ways to build simplicial complexes. We present here a few classical examples that are widely used in practice. 

A first example, is an immediate extension of the notion of $\alpha$-neighboring graph. Assume that we are given a set of points $\X$ in a metric space $(M, \rho)$
and a real number $\alpha \geq 0$. The \emph{Vietoris-Rips complex} $\Rips_\alpha(\X)$ is the set of simplices
$[x_0,\ldots,x_k]$ such that $d_\X(x_i,x_j)\leq\alpha$ for all $(i,j)$. It follows immediately from the definition that this is an abstract simplicial complex. However, in general, even when $\X$ is a finite subset of $\R^d$, $\Rips_\alpha(\X)$ does not admit a geometric realization in $\R^d$; in particular, it can be of dimension higher than $d$.

\begin{figure}[h]
	\centering
		\includegraphics[width = 0.8 \columnwidth]{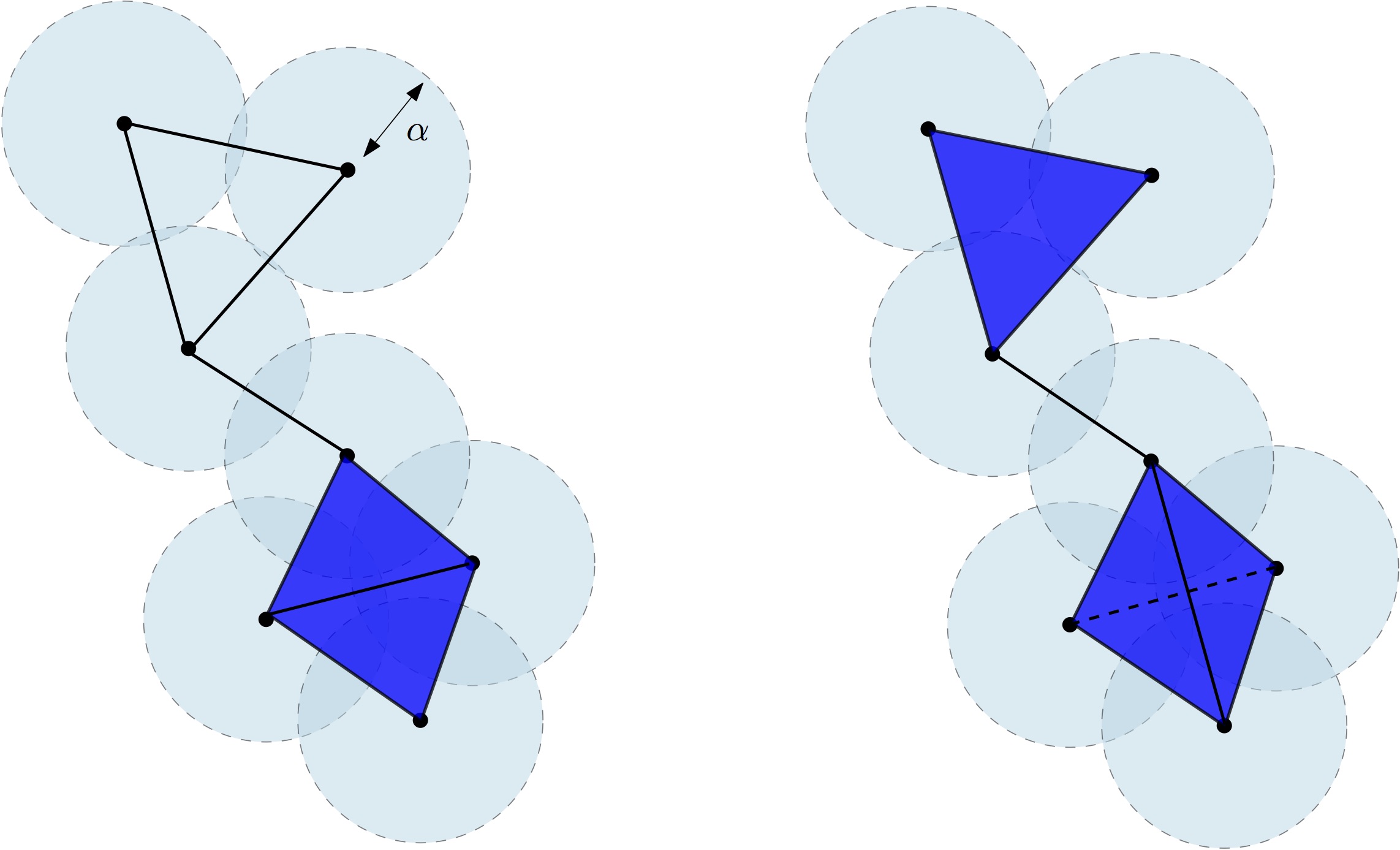} 
		\caption{The \v Cech complex $\Cech_\alpha(\X)$ (left) and the and Vietoris-Rips $\Rips_{2 \alpha}(\X)$ (right) of a finite point cloud in the plane $\R^2$. The bottom part of $\Cech_\alpha(\X)$ is the union of two adjacent triangles, while the bottom part of $\Rips_{2 \alpha}(\X)$ is the tetrahedron spanned by the four vertices and all its faces. The dimension of the \v Cech complex is $2$. The dimension of the Vietoris-Rips complex is $3$. Notice that this later is thus not embedded in $\R^2$.} 
	\label{fig:RipsCech}
\end{figure}

Closely related to the Vietoris-Rips complex is the \emph{\v Cech complex} $\Cech_\alpha(\X)$ that is defined as
the set of simplices $[x_0,\ldots,x_k]$ such that the $k+1$ closed balls
$B(x_i,\alpha)$ have a non-empty intersection. 
Notice that these two complexes are related by
$$\Rips_\alpha(\X)\subseteq\Cech_\alpha(\X)\subseteq\Rips_{2\alpha}(\X)$$
and that, if $\X \subset \R^d$ then $\Cech_\alpha(\X)$ and $\Rips_{2\alpha}(\X)$ have the same $1$-dimensional skeleton, i.e. the same set of vertices and edges.
%Note also that these two families of complexes only depend on the pairwise distances between the points of $\X$.

\paragraph{The nerve theorem.}
The \v Cech complex is a particular case of a family of complexes associated to covers. Given a {\em cover} $\mathcal{U} = (U_i)_{i \in I}$ of $\M$, i.e. a family of sets $U_i$ such that $\M = \cup_{i \in I} U_i$, the {\em nerve of $\mathcal{U}$} is the abstract simplicial complex $C(\mathcal{U})$ whose vertices are the $U_i$'s and such that
$$\sigma = [U_{i_0}, \cdots, U_{i_k}] \in C(\mathcal{U}) \ \mbox{\rm if and only if} \ \bigcap_{j=0}^k U_{i_j} \not = \emptyset.$$
Given a cover of a data set, where each set of the cover can be, for example, a local cluster or a grouping of data points sharing some common properties, its nerve provides a compact and global combinatorial description of the relationship between these sets through their intersection patterns - see Figure \ref{fig:nerve}. 
%In other words, the nerve of the cover provides a compact description of the global structure of the data when the points that are in the same set have been identified. 

\begin{figure}[h]
	\centering
		\includegraphics[width = 0.8 \columnwidth]{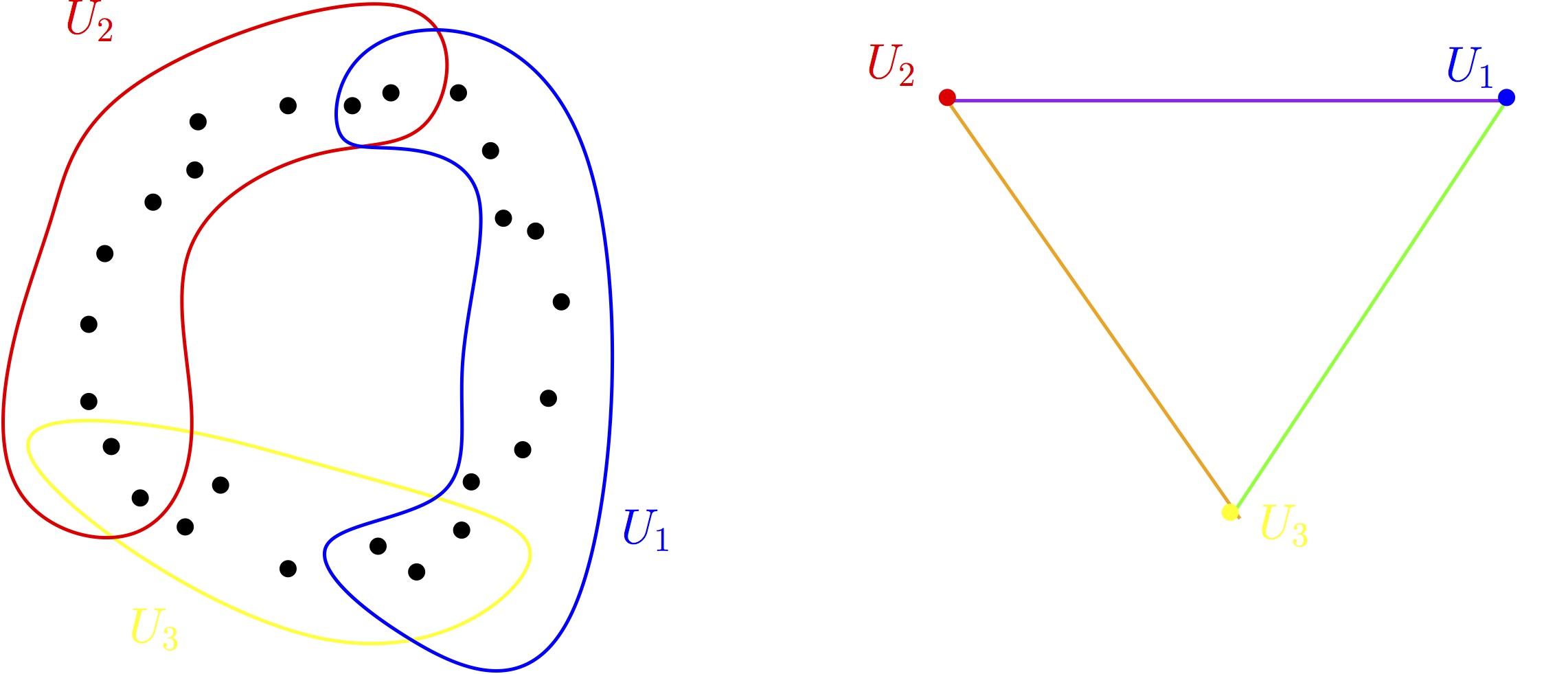} 
		\caption{The nerve of a cover of a set of sampled points in the plane.}
	\label{fig:nerve}
\end{figure}

A fundamental theorem in algebraic topology, relates, under some assumptions, the topology of the nerve of a cover to the topology of the union of the sets of the cover. To be formally stated, this result, known as the Nerve Theorem, requires to introduce a few notions. 

Two topological spaces $X$ and $Y$ are usually considered as being the same from a topological point of view if they are {\em homeomorphic}, i.e. if there exist two continuous bijective maps $f: X \to Y$ and $g: Y \to X$ such that $f \circ g$ and $g \circ f$ are the identity map of $Y$ and $X$ respectively.  
In many cases, asking $X$ and $Y$ to be homeomorphic turns out to be a too strong requirement to ensure that $X$ and $Y$ share the same topological features of interest for \tda. 
Two continuous maps $f_0, f_1: X \to Y$ are said to be {\em homotopic} is there exists a continuous map $H : X \times [0,1] \to Y$ such that for any $x \in X$, $H(x,0) = f_0(x)$ and $H(x,1) = g(x)$. The spaces $X$ and $Y$ are then said to be {\em homotopy equivalent} if there exist two maps $f: X \to Y$ and $g: Y \to X$ such that $f \circ g$ and $g \circ f$ are homotopic to the identity map of $Y$ and $X$ respectively. The maps $f$ and $g$ are then called {\em homotopy equivalent}. The notion of homotopy equivalence is weaker than the notion of homeomorphism: if $X$ and $Y$ are homeomorphic then they are obviously homotopy equivalent, the converse being not true. However, spaces that are homotopy equivalent still share many topological invariant, in particular they have the same homology - see Section \ref{sec:geometric-inference}. \\
A space is said to be {\em contractible} if it is homotopy equivalent to a point. Basic examples of contractible spaces are the balls and, more generally, the convex sets in $\R^d$. Open covers whose all elements and their intersections are contractible have the remarkable following property. 

\begin{theorem}[Nerve theorem] \label{thm:nerve}
Let $\mathcal{U} = (U_i)_{i \in I}$ be a cover of a topological space $X$ by open sets such that the intersection of any subcollection of the $U_i$'s is either empty or contractible. Then, $X$ and the nerve $C(\mathcal{U})$ are homotopy equivalent. 
\end{theorem}

It is easy to verify that convex subsets of Euclidean spaces are contractible. As a consequence, if $\mathcal{U} = (U_i)_{i \in I}$ is a collection of convex subsets of $\R^d$ then $C(\mathcal{U})$ and $\cup_{i\in I} U_i$ are homotopy equivalent. In particular, if $\X$ is a set of points in $\R^d$, then the \v Cech complex $\Cech_\alpha(\X)$ is homotopy equivalent to the union of balls $\cup_{x \in \X} B(x,\alpha)$. 

The Nerve Theorem plays a fundamental role in \tda: it provide a way to encode the topology of continuous spaces into abstract combinatorial structures that are well-suited for the design of effective data structures and algorithms.

%%%%%%%%%%%%%%%%%%%%%%%%%%%%%%%%%%%%%%%%%%%%%%%%%%%%%
%%%%%%%%%%%%%%%%%%%%%%%%%%%%%%%%%%%%%%%%%%%%%%%%%%%%%

\section{Using covers and nerves for exploratory data analysis and visualization: the Mapper algorithm} \label{sec:mapper}

Using the nerve of covers as a way to summarize, visualize and explore data is a natural idea that was first proposed for \tda\ in \cite{singh2007topological}, giving rise to the so-called Mapper algorithm. 

\begin{definition}
Let $f : X \to \R^d$, $d \geq 1$, be a continuous real valued function and let $\mathcal{U} = (U_i)_{i \in I}$ be a cover of $\R^d$. The pull back cover of $X$ induced by $(f,\mathcal{U})$ is the collection of open sets $(f^{-1}(U_i))_{i \in I}$. The refined pull back is the collection of connected components of the open sets $f^{-1}(U_i)$, $i \in I$. 
\end{definition}

\begin{figure}[h]
	\centering
		\includegraphics[width = 1 \columnwidth]{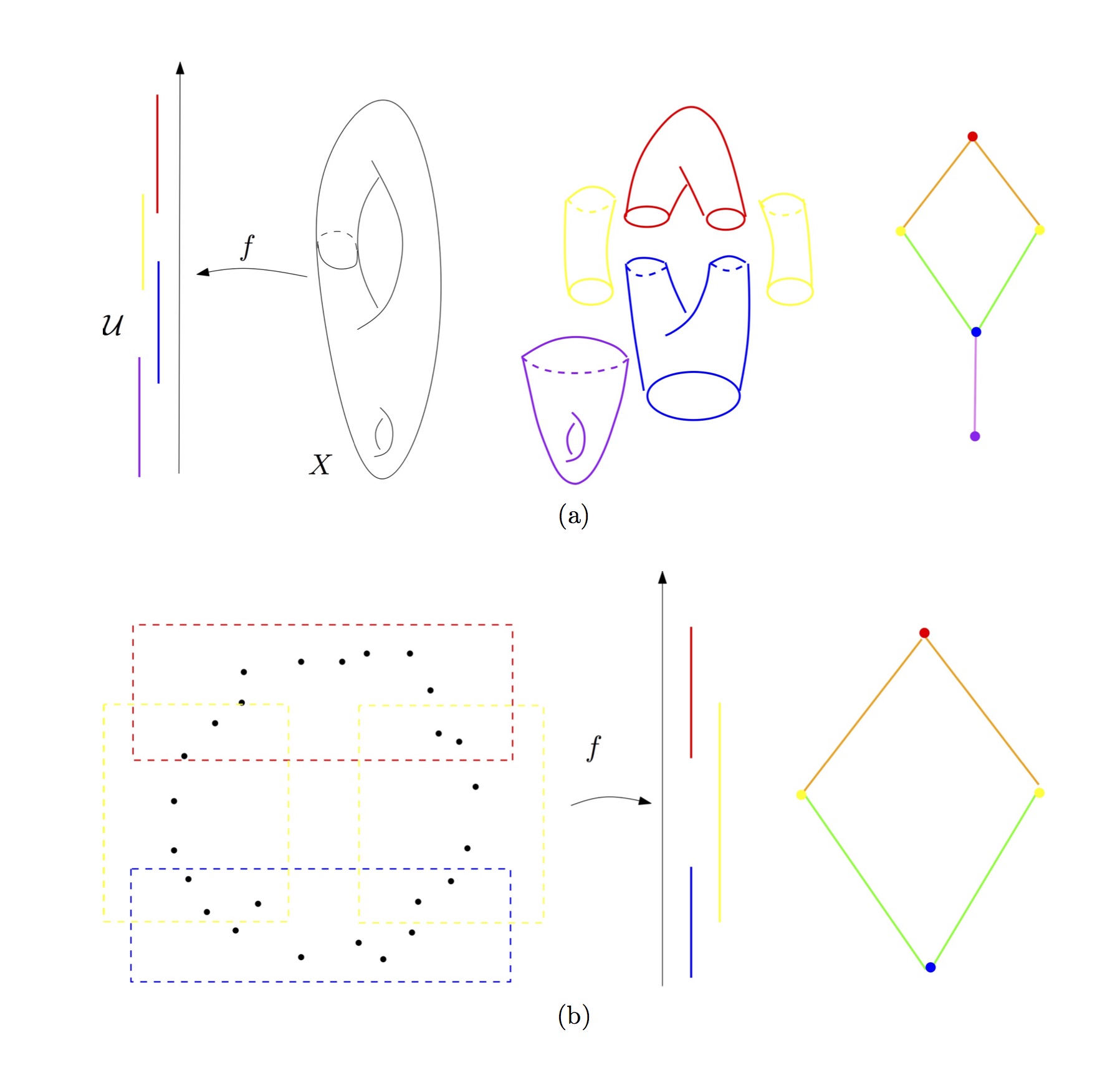} 
		\caption{(a) The refined pull back cover of the height function on a surface in $\R^3$ and its nerve. (b) The mapper algorithm on a point cloud sampled around a circle.}
	\label{fig:mapper}
\end{figure}

The idea of the Mapper algorithm is, given a data set $\X$ and well-chosen real valued function $f: \X \to \R^d$, to summarize $\X$ through the nerve of the refined pull back  of a cover $\mathcal{U}$ of $f(\X)$, see Figure~\ref{fig:mapper}(a) . For well-chosen covers $\mathcal{U}$ (see below), this nerve is a graph providing an easy and convenient way to visualize the summary of the data. 
It is described in Algorithm \ref{alg:mapper} and illustrated on a simple example in Figure~\ref{fig:mapper}(b). 

The Mapper algorithm is very simple but it raises several questions about the various choices that are left to the user and that we briefly discuss in the following.

\begin{algorithm}[ht]
  \caption{The Mapper algorithm}
  \label{alg:mapper}
  \begin{algorithmic}
    \STATE{\rm \bf Input:} A data set $\X$ with a metric or a dissimilarity measure between data points, a function $f : \X \to \R$ (or $\R^d$), and a cover $\mathcal{U}$ of $f(\X)$.
    \STATE  for each $U \in \mathcal{U}$, decompose $f^{-1}(U)$ into clusters $C_{U,1}, \cdots, C_{U,k_U}$.
    \STATE  Compute the nerve of the cover of $X$ defined by the $C_{U,1}, \cdots, C_{U,k_U}$, $U \in \mathcal{U}$
    \STATE{\rm \bf Output:}  a simplicial complex, the nerve (often a graph for well-chosen covers $\to$ easy to visualize): \\
- a vertex $v_{U,i}$ for each cluster $C_{U,i}$,\\
- an edge between $v_{U,i}$ and $v_{U',j}$ iff $C_{U,i} \cap C_{U',j} \not = \emptyset$
  \end{algorithmic}
\end{algorithm}

\paragraph{The choice of $f$.} The choice of the function $f$, sometimes called the filter or lens function, strongly depends on the features of the data that one expect to highlight. The following ones are among the ones more or less classically encountered in the literature:
\begin{itemize}
\item[-] Density estimates: the mapper complex may help to understand the structure and connectivity of high density areas (clusters). 
\item[-] PCA coordinates or coordinates functions obtained from a non linear dimensionality reduction (NLDR) technique, eigenfunctions of graph laplacians,... may help to reveal and understand some ambiguity in the use of non linear dimensionality reductions. 
\item[-] The centrality function $f(x) = \sum_{y \in \X} d(x,y)$ and the eccentricity function $f(x) = \max_{y\in \X} d(x,y)$, appears sometimes to be good choices that do not require any specific knowledge about the data. 
\item[-] For data that are sampled around 1-dimensional filamentary structures, the distance function to a given point allows to recover the underlying topology of the filamentary structures \cite{Chazal:2015:GAF:2767938.2767967}.
\end{itemize}

\paragraph{The choice of the cover $\mathcal{U}$.} When $f$ is a real valued function, a standard choice is to take $\mathcal{U}$ to be a set of regularly spaced intervals of equal length $r >0$ covering the set $f(\X)$. The real $r$ is sometimes called the {\em resolution} of the cover and the percentage $g$ of overlap between two consecutive intervals is called the the {\em gain} of the cover.
%- see Figure \ref{fig:resolution-gain}.
Note that if the gain $g$ is chosen below $50\%$, then every point of the real line is covered by at most 2 open sets of $\mathcal{U}$ and the output nerve is a graph. 
It is important to notice that the output of the Mapper is very sensitive to the choice of $\mathcal{U}$ and small changes in the resolution and gain parameters may results in very large changes in the output, making the method very unstable.  A classical strategy consists in exploring some range of parameters and select the ones that turn out to provide the most informative output from the user perspective. 

%\begin{figure}[h]
%	\centering
%		\includegraphics[width = 0.8 \columnwidth]{resolution-gain.jpg} 
%		\caption{An example of a cover of the real line with resolution $r$ and the gain $g = 25\%$.}
%	\label{fig:resolution-gain}
%\end{figure} 

\paragraph{The choice of the clusters.} The Mapper algorithm requires to cluster the preimage of the open sets $U \in \mathcal{U}$. There are two strategies to compute the clusters. A first strategy consists in applying, for each $U \in \mathcal{U}$, a cluster algorithm, chosen by the user, to the premimage $f^{-1}(U)$. A second, more global, strategy consists in building a neighboring graph on top of the data set $\X$, e.g. k-NN graph or $\eps$-graph, and, for each $U \in \mathcal{U}$, taking the connected components of the subgraph with vertex set $f^{-1}(U)$.

\paragraph{Theoretical and statistical aspects of Mapper.}
Based on the results on stability and the structure of Mapper proposed in \cite{carriere2015structure}, advances towards a statistically well-founded version of Mapper have been obtained recently in~\cite{carriere2018statistical}. Unsurprisingly, the convergence of Mapper depends on both the sampling of the data and the regularity of the filter function. Moreover, subsampling strategies can be proposed to select a complex in a Rips filtration at a convenient scale, as well as the  resolution and the gain for defining the Mapper graph. The case of stochastic and multivariate filters has also been studied in \cite{carriere2019approximation}. % Confidence regions are also provided for the topological features extracted from the Mapper graph.
An alternative description of the probabilistic convergence of Mapper, in term of categorification, has also been proposed in \cite{brown2020probabilistic}. Other approaches have been proposed to study and deal with the instabilities of the Mapper algorithm in \cite{dey2016multiscale, dey2017topological}.  

\paragraph{Data Analysis with Mapper.} 

As an exploratory data analysis tool, Mapper has been successfully used for clustering and feature selection.
The idea is to identify specific structures in the Mapper graph (or complex), in particular loops and flares. 
%\fred{Il faut qu'on trouve une bonne figure à mettre ici!!!!}. 
These structures are then used to identify interesting clusters or to select features or variable that best discriminate the data in these structures. Applications on real data, illustrating these techniques, may be found, for example, in \cite{yao2009topological,lum2013extracting,carriere2020topological}.
 
%\fred{Pas trop envie de detailler une ou des applis... Plutot donner un ou deux papiers o\`u les applis sont assez convaincantes et laisser le lecteur se debrouiller avec. On peut egalement illustrer tout cela dans la section soft en decorticant un jeu de donnees...}
%\begin{itemize}
%\item[-] 
%\end{itemize}

%\fred{Penser à lister des papiers avec des applis convaincantes}
%\fred{Mettre quelque chose sur le soft!}

%\bert{Les papiers de la clic Ohio (Wang Memoli Dey)}\fred{je mets une phrase dans le paragraphe statistical mapper : ce n'est pas vraiment stat mais ça me semble être le meilleur endroit. Peut-etre renommer le paragraphe en ''Theoretical and statistical aspects of Mapper}

%%%%%%%%%%%%%%%%%%%%%%%%%%%%%%%%%%%%%%%%%%%%%%%%%%%%%%%%%%%%
%%%%%%%%%%%%%%%%%%%%%%%%%%%%%%%%%%%%%%%%%%%%%%%%%%%%%%%%%%%%

\section{Geometric reconstruction and homology inference} \label{sec:geometric-inference}

%\fred{TODO : relire attentivement}
%\fred{Une phrase sur la question de l'estimation du reach - cf papier EJS sur le sujet}
%\fred{TODO : probleme de la numerotation des sous -sections (4.1 commence au bout de 3 pages)}

Another way to build covers and use their nerves to exhibit the topological structure of data is to consider union of balls centered on the data points. 
In this section, we assume that $\X_n = \{ x_0,\cdots, x_n \}$ is a subset of $\R^d$ sampled i.i.d. according to a probability measure $\mu$ with compact support $M \subset \R^d$. 
The general strategy to infer topological information about $M$ from $\mu$ proceeds in two steps that are discussed in the following of this section:
\begin{enumerate}
\item $\X_n$ is covered by a union of balls of fixed radius centered on the $x_i$'s. Under some regularity assumptions on $M$, one can relate the topology of this union of balls to the one of $M$;
\item From a practical and algorithmic perspective, topological features of $M$ are inferred from the nerve of the union of balls, using the Nerve Theorem. 
\end{enumerate}

In this framework, it is indeed possible to compare spaces through {\em isotopy equivalence}, a stronger notion than homeomorphism: $X \subseteq \R^d$ and $Y \subseteq \R^d$ are said to be {\em (ambient) isotopic} if there exists a continuous family of homeomorphisms $H : [0,1] \times \R^d \to \R^d$, $H$ continuous, such that for any $t \in [0,1]$, $H_t = H(t,.) : \R^d \to \R^d$ is an homeomorphism, $H_0$ is the identity map in $\R^d$ and $H_1(X) = Y$. Obviously, if $X$ and $Y$ are isotopic, then they are homeomorphic. The converse is not true: a knotted and an unknotted circles in $\R^3$ are not homeomorphic (notice that although this claim seems rather intuitive, its formal proof requires the use of some non obvious algebraic topology tools). 

\subsection{Distance-like functions and reconstruction} \label{sec:distance-like}
Given a compact subset $K$ of $\R^d$, and a non negative real number $r$, the union of balls of radius $r$ centered on $K$, $K^r = \cup_{x \in K} B(x,r)$, called the $r$-offset of $K$, is the $r$-sublevel set of the distance function $d_K : \R^d \to \R$ defined by $d_K(x) = \inf_{y \in K} \| x-y\|$; in other words, $K^r = d_k^{-1}([0,r])$. This remark allows to use differential properties of distance functions and to compare the topology of the offsets of compact sets that are close to each other with respect to the Hausdorff distance. 

\begin{definition}[Hausdorff distance in $\R^d$]
The Hausdorff distance between two compact subsets $K, K'$ of $\R^d$ is defined by
$$d_H(K,K') = \| d_K - d_{K'} \|_\infty = \sup_{x \in \R^d} |d_K(x) - d_{K'}(x)|.$$
\end{definition}

In our setting, the considered compact sets are the data set $\X_n$ and of the support $M$ of the measure $\mu$. When $M$ is a smooth compact submanifold, under mild conditions on $d_H(\X_n,M)$, for some well-chosen $r$, the offsets of $\X_n$ are homotopy equivalent to $M$, \cite{niyogi2008finding,cl-smrnn-08} - see Figure \ref{fig:offset-torus} for an illustration. These results extend to larger classes of compact sets and leads to stronger results on the inference of the isotopy type of the offsets of $M$, \cite{chazal2009sampling,chazal2008nca}. They also lead to results on the estimation of other geometric and differential quantities such as normals \cite{chazal2008nca}, curvatures \cite{chazal2008scm} or boundary measures  \cite{chazal2007sbm} under assumptions on the Haussdorff distance between the underlying shape and the data sample. 

\begin{figure}[h]
	\centering
		\includegraphics[width = 0.6 \columnwidth]{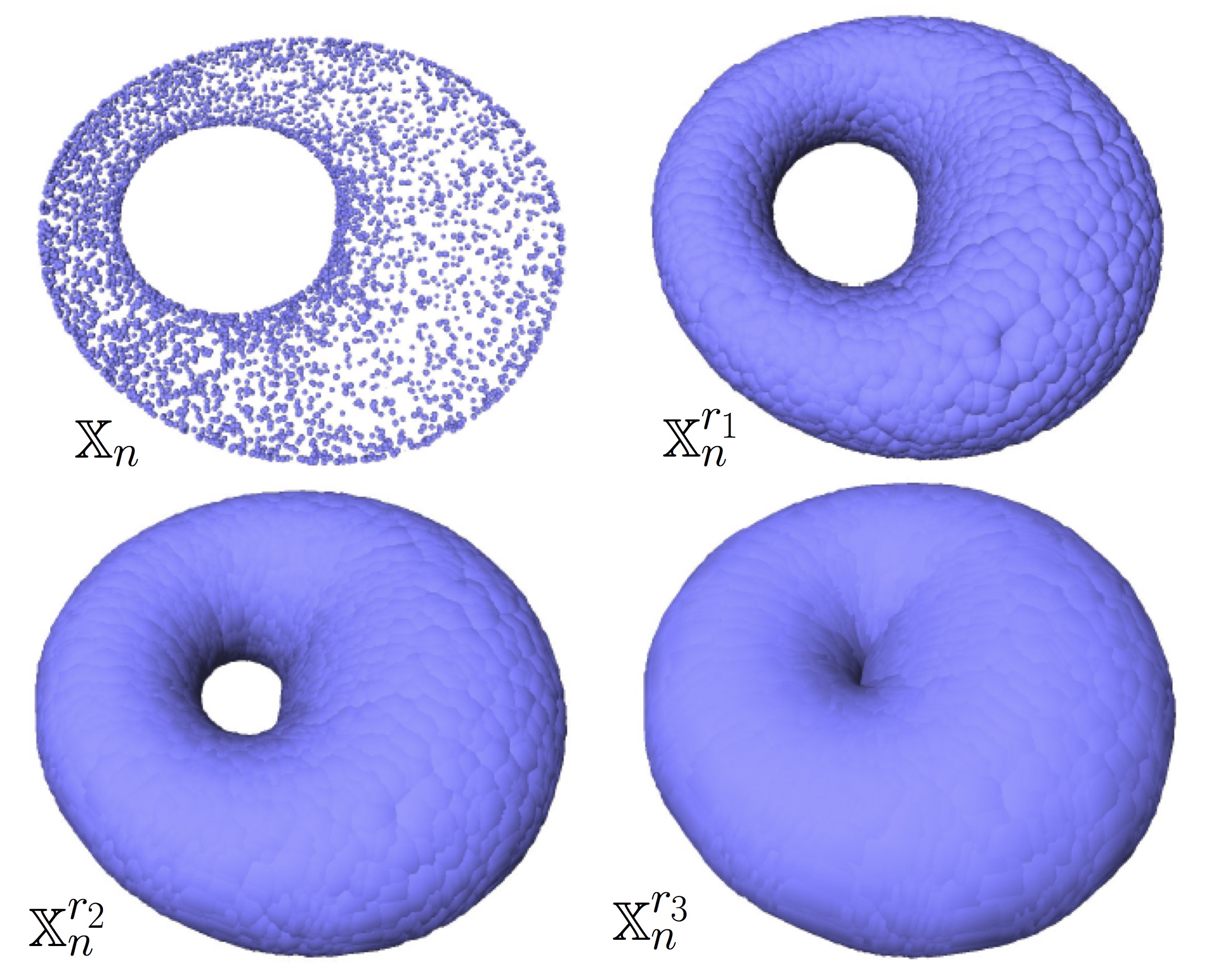} 
		\caption{The example of a point cloud $\X_n$ sampled on the surface of a torus in $\R^3$ (top left) and its offsets for different values of radii $r_1 < r_2 < r_3$. For well chosen values of the radius (e.g. $r_1$ and $r_2$), the offsets are clearly homotopy equivalent to a torus.}
	\label{fig:offset-torus}
\end{figure}  

These results rely on the $1$-semiconcavity of the squared distance function $d_K^2$, i.e. the convexity of the function $x \to \| x \|^2 -d_K^2(x)$, and can be naturally stated in the following general framework. 

\begin{definition}
A function $\phi: \R^d\to\R_+$ is {\em distance-like} if it is proper (the pre-image of any compact set in $\R$ is a compact set in $\R^d$) and $x \to \| x \|^2 - \phi^2(x)$ is convex.
\end{definition} 

Thanks to its semiconcavity, a distance-like function $\phi$ have a well-defined, but not continuous, gradient $\nabla \phi : \R^d \to \R^d$ that can be integrated into a continuous flow \citep{petrunin} that allows to track the evolution of the topology of its sublevel sets and to compare it to the one of the sublevel sets of close distance-like functions.  

\begin{definition}
Let $\phi$ be a distance-like function and let $\phi^r = \phi^{-1}([0,r])$ be the $r$-sublevel set of $\phi$.
\begin{itemize}
\item A point $x \in \R^d$ is called \emph{$\alpha$-critical} if $\| {\nabla_x \phi} \| \leq \alpha$. The corresponding value $r = \phi(x)$ is also said to be \emph{$\alpha$-critical}.
\item The \emph{weak feature size} of $\phi$ at $r$ is the minimum $r'
  > 0$ such that $\phi$ does not have any critical value between $r$
  and $r+r'$.  We denote it by $\wfs_\phi(r)$.
For any $0 < \alpha < 1$, the \emph{$\alpha$-reach} of $\phi$ is
  the maximum $r$ such that $\phi^{-1}((0,r])$ does not contain any
  $\alpha$-critical point.
  %Obviously, the $\alpha$-reach is always a lower bound for the weak-feature size, with $r=0$.
\end{itemize}
\end{definition}

The weak feature size $\wfs_\phi(r)$ (resp. $\alpha$-reach) measures the regularity of $\phi$ around its $r$-level sets (resp. $O$-level set). When $\phi = d_K$ is the distance function to a compact set $K \subset \R^d$, the $1$-reach coincides with the classical reach from geometric measure theory \cite{federer1959curvature}. Its estimation from random samples has been studied in \cite{aamari2019estimating}. 
An important property of a distance-like function $\phi$ is the topology of their sublevel sets $\phi^r$ can only change when $r$ crosses a $0$-critical value. 

\begin{lemma}[Isotopy Lemma \cite{grove1993cpt}] \label{prop:isotopy}
Let $\phi$ be a distance-like function and $r_1< r_2$ be two positive
numbers such that $\phi$ has no $0$-critical point, i.e. points $x$ such that $\nabla \phi(x) = 0$, in the subset
$\phi^{-1}([r_1,r_2])$. Then all the sublevel sets $\phi^{-1}([0,r])$ are
isotopic for $ r \in [r_1,r_2]$.
\end{lemma}

As an immediate consequence of the Isotopy Lemma, all the sublevel sets of $\phi$ between $r$ and $r+\wfs_\phi(r)$ have the same topology.
Now the following reconstruction theorem from \cite{chazal2011geometric} provides a connection between the topology of the sublevel sets of close distance-like functions. 

\begin{theorem}[Reconstruction Theorem] \label{thm:distance-like-reconstruction}
Let $\phi, \psi$ be two distance-like functions such that $\| \phi - \psi \|_\infty < \eps$, with
$\reach_\alpha (\phi) \geq R$ for some positive $\eps$ and $\alpha$. Then, for every
$r \in [4\eps/\alpha^2, R - 3\eps]$ and every $\eta\in (0,R)$, the sublevel
sets $\psi^r$ and $\phi^\eta$ are homotopy equivalent when
$$\eps \leq  \frac{R}{5 + 4/\alpha^2} .$$
\end{theorem}

Under similar but slightly more technical conditions the Reconstruction Theorem can be extended to prove that the sublevel sets are indeed homeomorphic and even isotopic \cite{chazal2008nca,chazal2008scm}.

Coming back to our setting, and taking for $\phi = d_M$ and $\psi = d_{\X_n}$ the distance functions to the support $M$ of the measure $\mu$ and to the data set $\X_n$, the condition $\reach_\alpha (d_M) \geq R$ can be interpreted as regularity condition on $M$\footnote{As an example, if $M$ is a smooth compact submanifold then $\reach_0 (\phi)$ is always positive and known as the reach of $M$ \cite{federer1959curvature}.}. The Reconstruction Theorem combined with the Nerve Theorem tell that, for well-chosen values of $r,\eta$, the $\eta$-offsets of $M$ are homotopy equivalent to the nerve of the union of balls of radius $r$ centered on $\X_n$, i.e the Cech complex $\Cech_r(\X_n)$. 

%More precisely, one has the following result.

From a statistical perspective, the main advantage of these results involving Hausdorff distance is that the estimation of the considered topological quantities boil down to support estimation questions that have been widely studied - see Section \ref{sec:stat-hom-inference}. 

\subsection{Homology inference} \label{sec:homology-inference}

The above results provide a mathematically well-founded framework to infer the topology of shapes from a simplicial complex built on top of an approximating finite sample. However, from a more practical perspective it raises raise two issues. First, the Reconstruction Theorem requires a regularity assumption through the $\alpha$-reach condition that may not always be satisfied and, the choice of a radius $r$ for the ball used to build the \v Cech complex $\Cech_r(\X_n)$. Second, $\Cech_r(\X_n)$ provides a topologically faithfull summary of the data, through a simplicial complex that is usually not well-suited for further data processing. 
One often needs easier to handle topological descriptors, in particular numerical ones, that can be easily computed from the complex. This second issue is addressed by considering the homology of the considered simplicial complexes in the next paragraph, while the first issue will be addressed in the next section with the introduction of persistent homology. 

\paragraph{Homology in a nutshell.}
Homology is a classical concept in algebraic topology providing a powerful tool to formalize and handle the notion of topological features of a topological space or of a simplicial complex in an algebraic way. For any dimension $k$, the $k$-dimensional ``holes'' are represented by a vector space $H_k$ whose dimension is intuitively the number of such independent features.  For example the 0-dimensional homology group $H_0$ represents the 
connected components of the complex, the $1$-dimensional homology group $H_1$ represents the 1-dimensional loops, the $2$-dimensional homology group $H_2$ represents the 2-dimensional cavities,... 
%See Hatcher \cite{h-at-01} for an introduction to simplicial homology.

To avoid technical subtleties and difficulties, we restrict the introduction of homology to the minimum that is necessary to understand its usage in the following of the paper. In particular we restrict to homology with coefficients in  $\Z_2$, i.e. the field with two elements $0$ and $1$ such that $1+1 = 0$, that turns out to be geometrically a little bit more intuitive. However, all the notions and results presented in the sequel naturally extend to homology with coefficient in any field. We refer the reader to \cite{h-at-01} for a complete and comprehensible introduction to homology and to \cite{ghristhomological} for a recent concise and very good introduction to applied algebraic topology and its connections to data analysis. 

Let $K$ be a (finite) simplicial complex and let $k$ be a non negative integer. The {\em space of $k$-chains} on $K$, $C_k(K)$ is the set whose elements are the formal (finite) sums of $k$-simplices of $K$. More precisely, if $\{ \sigma_1, \cdots \sigma_p \}$ is the set of $k$-simplices of $K$, then any $k$-chain can be written as 
\begin{equation*}
c = \sum_{i=1}^p \eps_i \sigma_i \ \ {\rm with} \ \ \eps_i \in \Z_2.
\end{equation*}
If $c' = \sum_{i=1}^p \eps_i' \sigma_i$ is another $k$-chain and $\lambda \in \Z_2$, the sum $c+c'$ is defined as $c+c' = \sum_{i=1}^p (\eps_i + \eps_i') \sigma_i$ and the product $\lambda.c$ is defined as $\lambda.c = \sum_{i=1}^p (\lambda.\eps_i) \sigma_i$, making $C_k(K)$ a vector space with coefficients in $\Z_2$. Since we are considering coefficient in $\Z_2$, geometrically a $k$-chain can be seen as a finite collection of $k$-simplices and the sum of two $k$-chains as the symmetric difference of the two corresponding collections\footnote{Recall that the symmetric difference of two sets $A$ and $B$ is the set $A \Delta B = (A \setminus B) \cup (B \setminus A)$.}. 

The {\em boundary} of a $k$-simplex $\sigma = [v_0, \cdots, v_k]$ is the $(k-1)$-chain 
\begin{equation*}
\partial_k (\sigma) = \sum_{i=0}^k  (-1)^i [v_0, \cdots, \hat v_i, \cdots, v_k]
\end{equation*}
where $[v_0, \cdots, \hat v_i, \cdots, v_k]$ is the $(k-1)$-simplex spanned by all the vertices except $v_i$ \footnote{Notice that as we are considering coefficients in $\Z_2$, here $-1 = 1$ and thus $(-1)^i = 1$ for any $i$.}. As the $k$-simplices form a basis of $C_k(K)$, $\partial_k$ extends as a linear map from $C_k(K)$ to $C_{k-1}(K)$ called the {\em boundary operator}. 
The kernel $Z_k(K) = \{ c \in C_k(K) : \partial_k = 0 \}$ of $\partial_k$ is called the {\em space of $k$-cycles of $K$} and the image $B_k(K) = \{ c \in C_k(K) : \exists c' \in C_{k+1}(K), \partial_{k+1}(c') = c \}$ of $\partial_{k+1}$ is called the {\em space of $k$-boundaries} of $K$. 
The boundary operators satisfy the fundamental following property: 
\begin{equation*}
\partial_{k-1} \circ \partial_k \equiv 0 \ \ \mbox{for any} \ \  k \geq 1. 
\end{equation*}
In other words, any $k$-boundary is a $k$-cycle, i.e. $B_k(K) \subseteq Z_k(K) \subseteq C_k(K)$. These notions are illustrated on Figure \ref{fig:CyclesBoundaries}.

\begin{figure}[h]
	\centering
		\includegraphics[width = 0.8 \columnwidth]{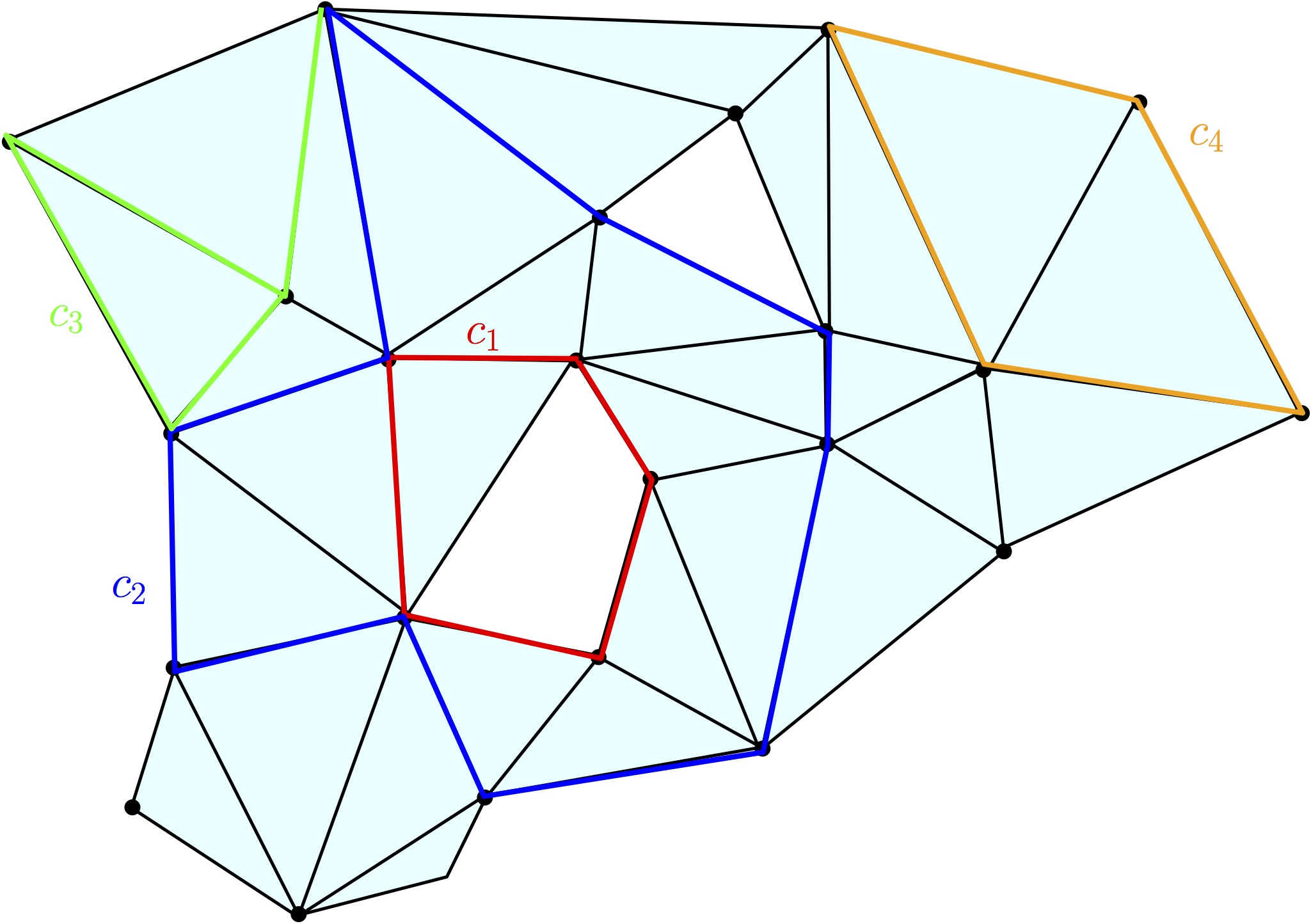} 
		\caption{Some examples of chains, cycles and boundaries on a $2$-dimensional complex $K$: $c_1, c_2$ and $c_4$ are $1$-cycles; $c_3$ si a $1$-chain but not a $1$-cycle; $c_4$ is the $1$-boundary, namely the boundary of the $2$-chain obtained as the sum of the two triangles surrounded by $c_4$; The cycles $c_1$ and $c_2$ span the same element in $H_1(K)$ as their difference is the $2$-chain represented by the union of the triangles surrounded by the union of $c_1$ and $c_2$.}
	\label{fig:CyclesBoundaries}
\end{figure}

\begin{definition}[Simplicial homology group and Betti numbers]
The $k^{th}$ (simplicial) homology group of $K$ is the quotient vector space
\begin{equation*}
H_k(K) = Z_k(K) / B_k(K). 
\end{equation*}
The $k^{th}$ Betti number of $K$ is the dimension $\beta_k(K) = \dim H_k(K)$ of the vector space $H_k(K)$. 
\end{definition}

Two cycles $c,c' \in Z_k(K)$ are said to be {\em homologous} if they differ by a boundary, i.e. is there exists a $(k+1)$-chain $d$ such that $c' = c + \partial_{k+1}(d)$. Two such cycles give rise to the same element of $H_k$. In other words, the elements of $H_k(K)$ are the equivalence classes of homologous cycles. 

Simplicial homology groups and Betti numbers are topological invariants: if $K,K'$ are two simplicial complexes whose geometric realizations are homotopy equivalent, then their homology groups are isomorphic and their Betti numbers are the same. 

\begin{figure}[h]
	\centering
		\includegraphics[width = 0.8 \columnwidth]{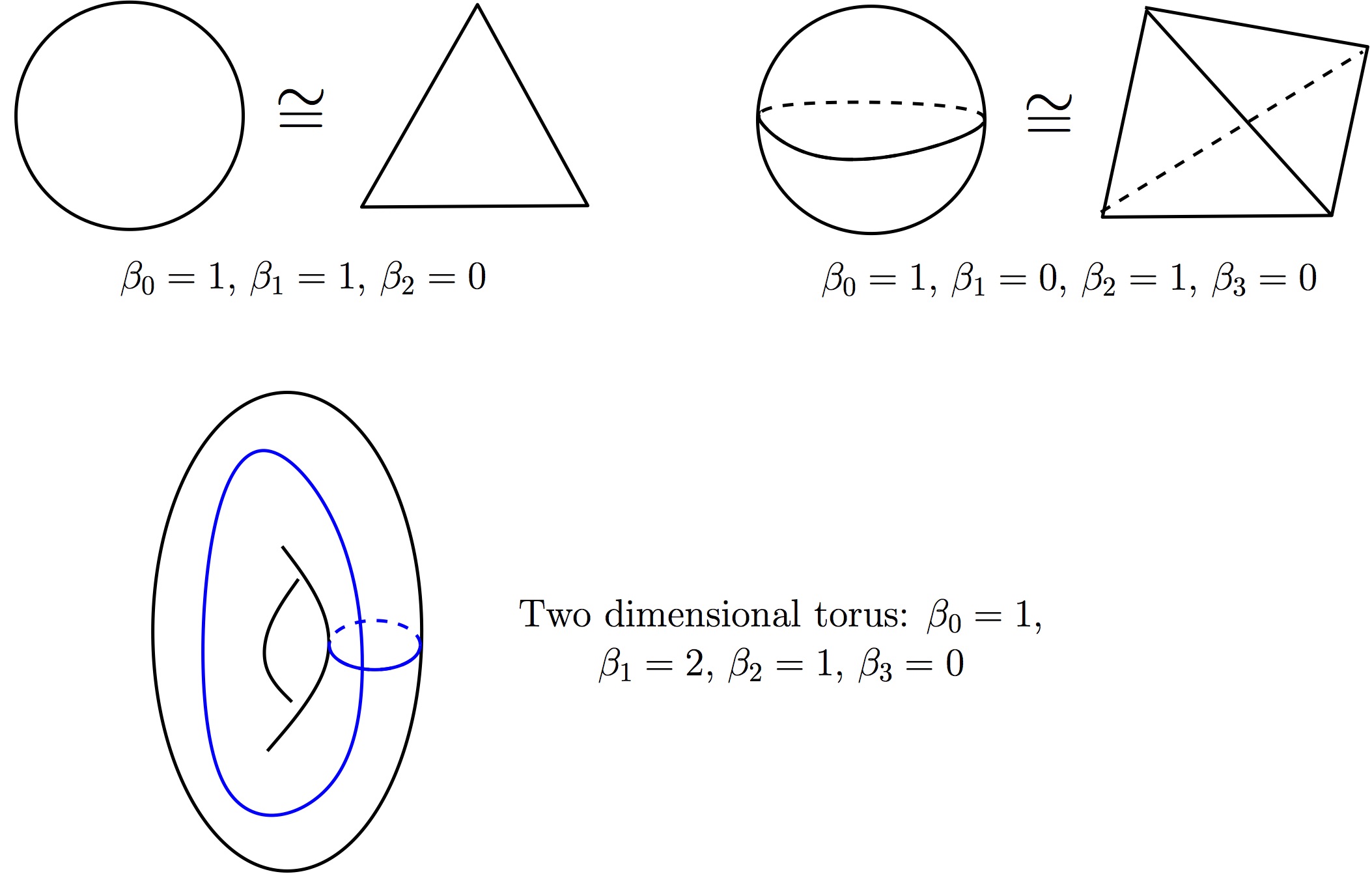} 
		\caption{The Betti numbers of the circle (top left), the 2-dimensional sphere (top right) and the 2-dimensional torus (bottom). The blue curves on the torus represent two independent cycles whose homology class is a basis of its 1-dimensional homology group.}
	\label{fig:ExampleBettiNb}
\end{figure}

{\em Singular homology} is another notion of homology that allows to consider larger classes of topological spaces. It is defined for any topological space $X$ similarly to simplicial homology except that the notion of simplex is replaced by the notion of {\em singular simplex} which is just any continuous map $\sigma : \Delta_k \to X$ where $\Delta_k$ is the standard $k$-dimensional simplex. The space of $k$-chains is the vector space spanned by the $k$-dimensional singular simplices and the boundary of a simplex $\sigma$ is defined as the (alternated) sum of the restriction of $\sigma$ to the $(k-1)$-dimensional faces of $\Delta_k$. 
A remarkable fact about singular homology it that it coincides with simplicial homology whenever $X$ is homeomorphic to the geometric realization of a simplicial complex.
This allows us, in the sequel of this paper, to indifferently talk about simplicial or singular homology for topological spaces and simplicial complexes. 

Observing, that if $f  : X \to Y$ is a continuous map, then for any singular simplex $\sigma : \Delta_k \to X$ in $X$, $f \circ \sigma : \Delta_k \to Y$ is a singular simplex in $Y$, one easily deduces that continuous maps between topological spaces canonically induce homomorphisms between their homology groups.
In particular, if $f$ is an homeomorphism or an homotopy equivalence, then it induces an isomorphism between $H_k(X)$ and $H_k(Y)$ for any non negative integer $k$. As an example, it follows from the Nerve Theorem that for any set of points $X \subset \R^d$ and any $r >0$ the $r$-offset $X^r$ and the \v Cech complex $\Cech_r(X)$ have isomorphic homology groups and the same Betti numbers. 

As a consequence, the Reconstruction Theorem \ref{thm:distance-like-reconstruction} leads to the following result on the estimation of Betti numbers. 

\begin{theorem} \label{thm:betti-inference}
Let $M \subset \R^d$ be a compact set such that $\reach_\alpha (d_M) \geq R >0$ for some $\alpha \in (0,1)$ and let $\X$ be a finite set of point such that $d_H(M,\X) = \eps < \frac{R}{5 + 4/\alpha^2}$. Then, for every $r \in [4\eps/\alpha^2, R - 3\eps]$ and every $\eta\in (0,R)$,  the Betti numbers of $\Cech_r(\X)$ are the same as the ones of $M^\eta$. \\
In particular, if $M$ is a smooth $m$-dimensional submanifold of $\R^d$, then $\beta_k(\Cech_r(\X)) = \beta_k(M)$ for any $k=0, \cdots, m$. 
\end{theorem}

From a practical perspective, this result raises three difficulties: first, the regularity assumption involving the $\alpha$-reach of $M$ may be too restrictive; second, the computation of the nerve of an union of balls requires they use of a tricky predicate testing the emptiness of a finite union of balls; third the estimation of the Betti numbers relies on the scale parameter $r$ whose choice may be a problem.  

To overcome these issues, \cite{co-tpbre-08} establishes the following result that offers a solution to the two first problems. 

\begin{theorem} \label{thm:betti-rips-inference}
Let $\M \subset \R^d$ be a compact set such that $\wfs(M) = \wfs_{d_M}(0) \geq R > 0$ and let $\X$ be a finite set of point such that $d_H(M,\X) = \eps < \frac{1}{9} \wfs(M)$. Then for any $r \in [2\eps , \frac{1}{4} ( \wfs(M) - \eps ) ]$ and any $\eta\in (0,R)$, 
$$\beta_k(X^\eta) = \mbox{\rm rk} \left ( H_k(\Rips_r(\X)) \to H_k(\Rips_{4r}(\X)) \right )$$
where rk$\left ( H_k(\Rips_r(\X)) \to H_k(\Rips_{4r}(\X)) \right )$ denotes the rank of the homomorphism induced by the (continuous) canonical inclusion $\Rips_r(\X) \hookrightarrow \Rips_{4r}(\X)$. 
\end{theorem}

Although this result leaves the question of the choice of the scale parameter $r$ open, it is proven in \cite{co-tpbre-08} that a multiscale strategy whose description is beyond the scope of this paper provides some help to identify the relevant scales at which Theorem \ref{thm:betti-rips-inference} can be applied. 

\subsection{Statistical aspects of Homology inference} \label{sec:stat-hom-inference}

According to the stability results  presented in the previous section, a statistical approach to topological inference is  strongly related  to  the problem of {\it distribution  support estimation} and {\it level sets estimation}  under the Hausdorff metric.  A large number of methods and results are available for estimating  the support of a distribution in statistics. For instance, the Devroye and Wise estimator \citep{DevroyeWise80} defined on a sample $\X_n$ is also a particular offset of $\X_n$. The convergence rates of  both $\X _n$ and  the Devroye and Wise estimator to the support of the distribution for the Hausdorff distance is studied in \citet{CuevasRCasal04} in $\R^d$.  More recently, the minimax rates of convergence of manifold  estimation for the Hausdorff metric, which is particularly relevant for topological inference, has been studied in \citet{GenoveseEtAl2012}. There is also a large literature about level sets estimation in various metrics \citep[see for instance][]{polonik1995measuring,tsybakov1997nonparametric,cadre2006kernel} and more particularly for the Hausdorff metric in \citet{chen2015density}. All these works about support and level sets estimation shine light on the statistical analysis of topological inference procedures.

In the paper \citet{niyogi2008finding}, it is  shown that the homotopy type of Riemannian manifolds with reach larger than a given constant  can be recovered with high probability from offsets of a sample on (or close to) the manifold. This  paper   was probably the first attempt to consider the topological inference problem in terms of probability. The result of \citet{niyogi2008finding} is derived from a retract contraction argument and on tight bounds over the  packing number of the manifold in order to control the Hausdorff distance between the manifold and the observed point cloud. The homology inference  in the noisy case, in the sense the distribution of the observation is concentrated around the manifold, is also studied in~\citet{niyogi2008finding,niyogi2011topological}. 
The assumption that the geometric object is a smooth Riemannian manifold is only used in the paper to control in probability the Hausdorff distance between the sample and the manifold, and is not actually necessary for the "topological part" of the result. Regarding the topological results, these are similar to those of \citet{chazal2009sampling,ChazalLeutier08} in the particular framework of Riemannian manifolds. 
Starting from  the result of \citet{niyogi2008finding}, the minimax rates of convergence of the homology type have been studied by \citet{balakrishnan12minimax} under various models, for Riemannian manifolds with reach larger than a constant.  In contrast,  a statistical version of \citet{chazal2009sampling}  has not yet been proposed.

More recently, following the ideas of \citet{niyogi2008finding}, \citet{bobrowski2014topological} have proposed a robust homology estimator for the level sets of both density and regression functions,  by considering  the inclusion map between nested pairs of estimated level sets (in the spirit of Theorem~\ref{thm:betti-rips-inference} above) obtained with a plug-in approach  from a kernel estimator. 

\begin{figure}[h]
	\centering
		\includegraphics[width = 0.8 \columnwidth]{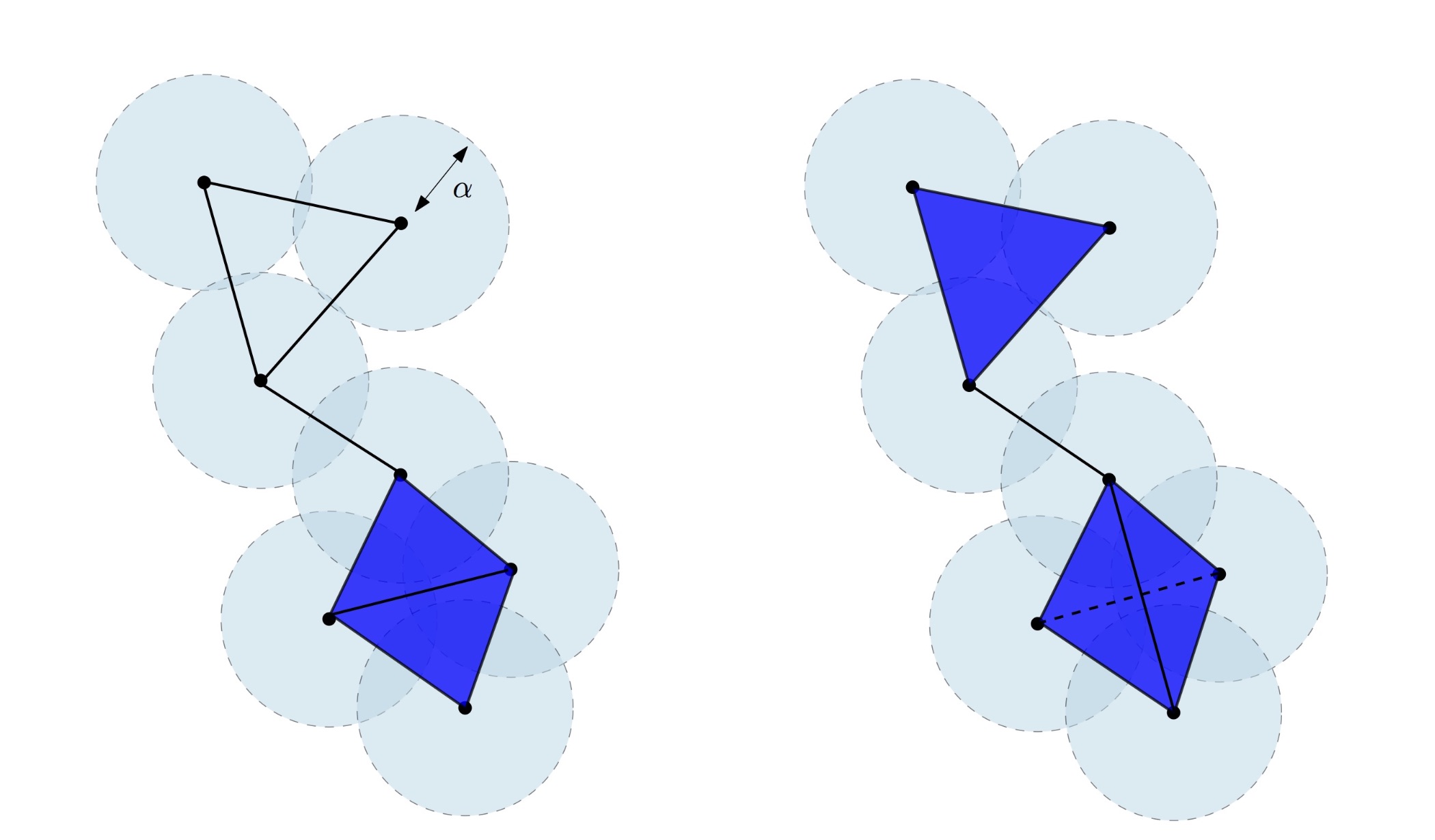}
		\caption{The effect of outliers on the sublevel sets of distance functions. Adding just a few outliers to a point cloud may dramatically change its distance function and the topology of its offsets.}
	\label{fig:outlierDistFn}
\end{figure} 

\subsection{Going beyond Hausdorff distance : distance to measure} \label{subsec:DTM}

It is well known that distance-based methods in \tda\ may fail completely in the presence of outliers. Indeed, adding even a single outlier to the point cloud  can change the distance function dramatically, see Figure \ref{fig:outlierDistFn} for an illustration. To answer this drawback, \citet{chazal2011geometric} have introduced an alternative distance function which is robust to noise,  the {\em distance-to-a-measure}. 

Given a probability distribution $P$ in $\R^d$ and a real parameter $0 \leq u \leq 1$, 
the notion of distance to the support of $P$ may be generalized as the function 
\begin{equation*}
\label{eq:DTMPrelimdef}
\delta_{P, u} : x \in \R^d  \mapsto \inf\{ t > 0 \, ; \,  P(B (x,t) ) \geq  u  \} ,
\end{equation*}
where $B (x,t)$ is the closed Euclidean ball of center $x$ and radius $t$. To avoid issues due to discontinuities of the map $P \to \delta_{P, u}$, the distance-to-measure function (DTM) with parameter $m \in [0,1]$ and power $r \geq 1$ is defined by 
\begin{equation}
\label{eq:DTMdef}
d_{P,m,r}(x)  : x \in \mathcal \R^d  \mapsto  \left(  \frac 1 {m} \int _0 ^{m} \delta_{P,u}^r(x)  d u \right)^{1/r}.
\end{equation}
%It was shown in \citet{chazal2011geometric}  that the DTM for $r=2$ shares many properties with classical distance functions that make it well-adapted for geometric inference purposes.
% It is recalled that the Wasserstein distance $W_r$ between two probability
%measures $P$ and $\tilde P$ on $\R^d$  is given by
%$$W_r(P,\tilde P)=\inf _{\pi \in \Pi(P,\tilde P)}\left(\int_{\R^d \times \R^d} \|x-y\|^r \pi (dx,dy)\right)^{\frac
%1r},$$
%where $\Pi(P,\tilde P)$ is the set of probability measures on ${ \R^d} \times {\R^d}$ with marginal distributions $P$ and $\tilde P$, see for instance \citet{RR98} or \citet{Vil08}. 
A nice property of the DTM proved in  \citet{chazal2011geometric} is its stability with respect to perturbations of $P$ in the Wasserstein metric. More precisely, the map $P \to d_{P,m,r}$ is $m^{- \frac 1 r}$-Lipschitz, i.e. if $P$ and $\tilde P$ are two probability distributions on $\R^d$, then
\begin{equation} \label{eq:DTMW}
 \|d_{P,m,r} - d_{\tilde P,m,r} \|_\infty \leq  m^{- \frac 1 r} W_r (P,\tilde P)
\end{equation}
where $W_r$ is the Wasserstein distance for the Euclidean metric on $\R^d$, with exponent $r$\footnote{See \cite{villani2003tot} for a definition of the Wasserstein distance}. %measures $P$ and $\tilde P$ on $\R^d$
This property implies that the DTM associated to close distributions in the Wasserstein metric have close sublevel sets. Moreover, when $r=2$, the function $d_{P,m,2}^2$ is semiconcave ensuring strong regularity properties on the geometry of its sublevel sets.  Using these properties, \citet{chazal2011geometric} show that, under general assumptions, if $\tilde P$ is a probability distribution approximating $P$, then the sublevel sets of $d_{\tilde P,m,2}$ provide a topologically correct approximation of the support of $P$. 

In practice, the measure $P$ is usually only known through a finite set of observations  $\X_n = \{ X_1,\dots, X_n\} $ sampled from $P$, raising the question of the approximation of the DTM.
A natural idea to estimate the DTM from $\X_n$ is to plug the empirical measure $P_n$ instead of $P$ in the definition of the DTM. This "plug-in strategy" corresponds to computing the distance to the empirical measure (DTEM). For $m = \frac k n $, the DTEM satisfies
  $$ d^r_{P_n,k/n,r}(x)  := \frac 1 k \sum _{j=1} ^k  {\|x -  \X_n \|}^r_{(j)} \,  ,$$
where $\|x -  \X_n \|_{(j)}$ denotes the distance between $x$ and its $j$-th neighbor in $\{X_1,\dots, X_n\}$.  This
quantity can be easily  computed in practice since it only requires the distances between $x$ and the sample points. The convergence of the DTEM to the DTM has been studied in \cite{cflmrw-rtidt-14} and \cite{cmm-rcrgi-15}. 

The introduction of DTM has motivated further works and applications in various directions such as topological data analysis \citep{buchet2015socg}, GPS traces analysis \citep{ccgjs-ddts-11}, density estimation \citep{biau2011weighted}, hypothesis testing \cite{brecheteau2019statistical}, clustering \citep{chazal2013persistence} just to name a few. Approximations, generalizations and variants of the DTM have also been considered in \citep{guibas2013witnessed,phillips2014goemetric,buchet2015efficient, brecheteau2020k}.

\section{Persistent homology} \label{sec:persistent-homology}
Persistent homology is a powerful tool to compute, study and encode efficiently  multiscale topological features of nested families of simplicial complexes and topological spaces. It does not only provide efficient algorithms to compute the Betti numbers of each complex in the considered families, as required for homology inference in the previous section, but also encodes the evolution of the homology groups of the nested complexes across the scales. 
Ideas and preliminary results underlying persistent homology theory trace back to the 20th century, in particular in \cite{frosini1992measuring, Barannikov1994TheFM, robins1999towards}. It started to know an important development in its modern form after the seminal works of \cite{elz-tps-02} and \cite{zc-cph-05}.

\subsection{Filtrations}

%\fred{Mentionner ici les filtrations DTM}

A {\em filtration of a simplicial complex} $K$ is a nested family of subcomplexes $(K_r)_{r \in T}$, where $T \subseteq \R$, such that for any $r, r' \in T$, if $r \leq r'$ then $K_r \subseteq K_{r'}$, and $K = \cup_{r \in T} K_r$. The subset $T$ may be either finite or infinite. More generally, a {\em filtration} of a topological space $\M$ is a nested family of subspaces $(M_r)_{r \in T}$, where $T \subseteq \R$, such that for any $r, r' \in T$, if $r \leq r'$ then $M_r \subseteq M_{r'}$ and, $M = \cup_{r \in T} M_r$. For example, if $f: \M \to \R$ is a function, then the family $M_r = f^{-1}((-\infty, r])$, $r \in \R$ defines a filtration called the sublevel set filtration of $f$. 

In practical situations, the parameter $r \in T$ can often be interpreted as a scale parameter and filtrations classically used in TDA often belong to one of the two following families. 

\paragraph{Filtrations built on top of data.} Given a subset $\X$ of a compact metric space $(M,\rho)$, the families of Rips-Vietoris complexes 
$( \Rips_r(\X) )_{r \in \R}$ and and \v Cech complexes $(\Cech_r(\X))_{r \in \R}$ are filtrations \footnote{we take here the convention that for $r<0$, $\Rips_r(\X) = \Cech_r(\X) = \emptyset$}. Here, the parameter $r$ can be interpreted as a resolution at which one considers the data set $\X$. For example, if $\X$ is a point cloud in $\R^d$, thanks to the Nerve theorem, the filtration $(\Cech_r(\X))_{r \in \R}$ encodes the topology of the whole family of unions of balls $\X^r = \cup_{x \in \X} B(x,r)$, as $r$ goes from $0$ to $+\infty$. As the notion of filtration is quite flexible, many other filtrations have been considered in the literature and can be constructed on top of data, such as e.g. the so called witness complex popularized in \tda\ by  \cite{DeSilva:2004:TEU:2386332.2386359}, the weighted Rips filtrations \cite{buchet2015efficient}, or the so-called DTM-filtrations \cite{anai2020dtm} that allow to handle data corrupted by noise and outliers.

\paragraph{Sublevel sets filtrations.} Functions defined on the vertices of a simplicial complex give rise to another important example of filtration: let $K$ be a simplicial complex with vertex set $V$ and $f: V \to \R$. Then $f$ can be extended to all simplices of $K$ by $f([v_0,\cdots, v_k]) = \max \{f(v_i): i=1,\cdots ,k \}$ for any simplex $\sigma = [v_0,\cdots, v_k] \in K$ and the family of subcomplexes $K_r = \{ \sigma \in K : f(\sigma) \leq r \}$ defines a filtration call the sublevel set filtration of $f$. Similarly, one can define the upperlevel set filtration of $f$.

In practice, even if the index set is infinite, all the considered filtrations are built on finite sets and are indeed finite. For example,  when $\X$ is finite, the Vietoris-Rips complex $\Rips_r(\X)$ changes only at a finite number of indices $r$. This allows to easily handle them from an algorithmic perspective.

\subsection{Starting with a few examples} \label{sec:pers-example}
Given a filtration $\Filt = ( F_r )_{r \in T}$ of a simplicial complex or a topological space, the homology of $F_r$ changes as $r$ increases: new connected components can appear, existing component can merge, loops and cavities can appear or be filled, etc... Persistent homology tracks these changes, identifies the appearing features and associates a life time to them. The resulting information is encoded as a set of intervals called a {\em barcode} or, equivalently, as a multiset of points in $\R^2$ where the coordinate of each point is the starting and end point of the corresponding interval. 

Before giving formal definitions, we introduce and illustrate persistent homology on a few simple examples. 

\paragraph{Example 1.}
Let $f : [0,1] \to \R$ be the function of Figure~\ref{fig:examples_pers_homology}(a) and let $(F_r = f^{-1}((-\infty, r]))_{r \in \R}$ be the sublevel set filtration of $f$.
All the sublevel sets of $f$ are either empty or a union of intervals, so the only non trivial topological information they carry is their $0$-dimensional homology, i.e. their number of connected components. For $r < a_1$, $F_r$ is empty, but at $r=a_1$ a first connected component appears in $F_{a_1}$. Persistent homology thus registers $a_1$ as the birth time of a connected component and start to keep track of it by creating an interval starting at $a_1$. Then, $F_r$ remains connected until $r$ reaches the value $a_2$ where a second connected component appears. Persistent homology starts to keep track of this new connected component by creating a second interval starting at $a_2$. Similarly, when $r$ reaches $a_3$, a new connected component appears and persistent homology creates a new interval starting at $a_3$. When $r$ reaches $a_4$, the two connected components created at $a_1$ and $a_3$ merges together to give a single larger component.
At this step, persistent homology follows the rule that this is the most recently appeared component in the filtration that dies: the interval started at $a_3$ is thus ended at $a_4$ and a first persistence interval encoding the lifespan of the component born at $a_3$ is created. When $r$ reaches $a_5$, as in the previous case, the component born at $a_2$ dies and the persistent interval $(a_2, a_5)$ is created. 
The interval created at $a_1$ remains until the end of the filtration giving rise to the persistent interval $(a_1, a_6)$ if the filtration is stopped at $a_6$, or $(a_1, +\infty)$ if $r$ goes to $+\infty$ (notice that in this later case, the filtration remains constant for $r > a_6$). The obtained set of intervals encoding the span life of the different homological features encountered along the filtration is called the {\em persistence barcode} of $f$. Each interval $(a,a')$ can be represented by the point of coordinates $(a,a')$  in $\R^2$ plane. The resulting set of points is called the {\em persistence diagram} of $f$. Notice that a function may have several copies of the same interval in its persistence barcode. As a consequence, the persistence diagram of $f$ is indeed a multi-set where each point has an integer valued multiplicity. Last, for technical reasons that will become clear in the next section, one adds to the persistence all the points of the diagonal $\Delta = \{ (b,d) : b=d \}$ with an infinite multiplicity.

\begin{figure}
	\centering
		\includegraphics[width = 1 \columnwidth]{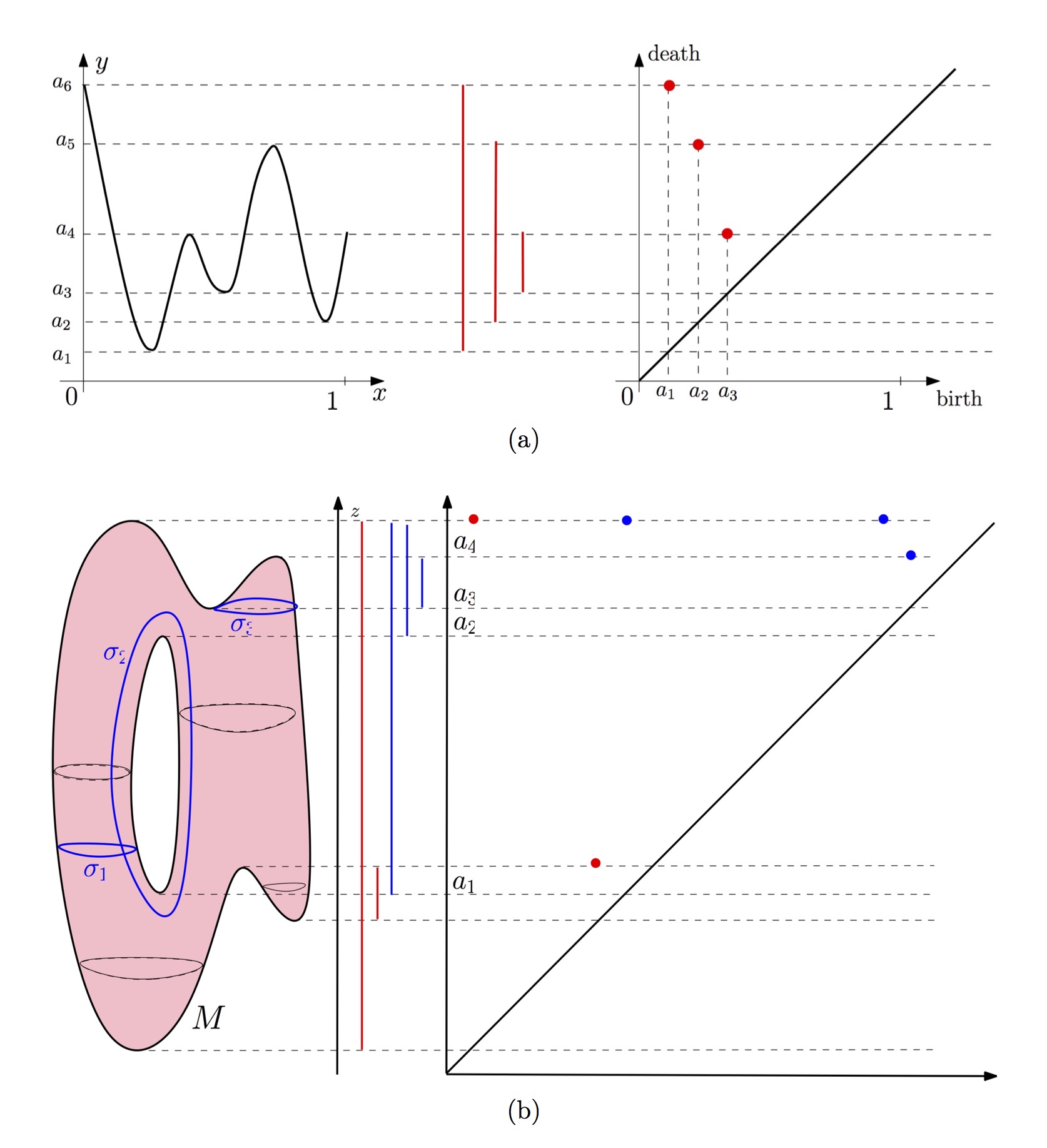} 
		\caption{(a) Example 1: The persistence barcode and the persistence diagram of a function $f : [0,1] \to \R$. (b) Example 2:  the persistence barcode and the persistence diagram of the height function (projection on the $z$-axis) defined on a surface in $\R^3$.}
	\label{fig:examples_pers_homology}
\end{figure} 

\paragraph{Example 2.}
Let now $f : M \to \R$ be the function of Figure \ref{fig:examples_pers_homology}(b) where $M$ is a 2-dimensional surface homeomorphic to a torus, and let $(F_r = f^{-1}((-\infty, r]))_{r \in \R}$ be the sublevel set filtration of $f$. The $0$-dimensional persistent homology is computed as in the previous example, giving rise to the red bars in the barcode. Now, the sublevel sets also carry $1$-dimensional homological features. When $r$ goes through the height $a_1$, the sublevel sets $F_r$ that were homeomorphic to two discs become homeomorphic to the disjoint union of a disc and an annulus, creating a first cycle homologous to $\sigma_1$ on Figure~\ref{fig:examples_pers_homology}(b). A interval (in blue) representing the birth of this new $1$-cycle is thus started at $a_1$. Similarly, when $r$ goes through the height $a_2$ a second cycle, homologous to $\sigma_2$ is created, giving rise to the start of a new persistent interval. These two created cycles are never filled (indeed they span $H_1(M)$) and the corresponding intervals remains until the end of the filtration. When $r$ reaches $a_3$, a new cycle is created that is filled and thus dies at $a_4$, giving rise to the persistence interval $(a_3,a_4)$. 
So, now, the sublevel set filtration of $f$ gives rise to two barcodes, one for $0$-dimensional homology (in red) and one for $1$-dimensional homology (in blue). As previously, these two barcodes can equivalently be represented as diagrams in the plane.

\paragraph{Example 3.}
In this last example we consider the filtration given by a union of growing balls centered on the finite set of points $C$ in Figure \ref{fig:Cech-diag-example}. Notice that this is the sublevel set filtration of the distance function to $C$, and thanks to the Nerve Theorem, this filtration is homotopy equivalent to the \v Cech filtration built on top of $C$. Figure \ref{fig:Cech-diag-example} shows several level sets of the filtration:\\ 
a) For the radius $r=0$, the union of balls is reduced to the initial finite set of point, each of them corresponding to a $0$-dimensional feature, i.e. a connected component; an interval is created for the {\em birth} for each of these features at $r=0$.\\ 
b) Some of the balls started to overlap resulting in the {\em death} of some connected components that get merged together; the persistence diagram keeps track of these deaths, putting an end point to the corresponding intervals as they disappear.\\ 
c) New components have merged giving rise to a single connected component and, so, all the intervals associated to a $0$-dimensional feature have been ended, except the one corresponding to the remaining components; two new $1$-dimensional features, have appeared resulting in two new intervals (in blue) starting at their birth scale.\\ 
d) One of the two 1-dimensional cycles has been filled, resulting in its death in the filtration and the end of the corresponding blue interval.\\ 
e) all the $1$-dimensional features have died, it only remains the long (and never dying) red interval. As in the previous examples, the final barcode can also be equivalently represented as a persistence diagram where every interval $(a,b)$ is represented by the the point of coordinate $(a,b)$ in $\R^2$.
Intuitively the longer is an interval in the barcode or, equivalently the farther from the diagonal is the corresponding point in the diagram, the more persistent, and thus relevant, is the corresponding homological feature across the filtration. Notice also that for a given radius $r$, the $k$-th Betti number of the corresponding union of balls is equal of the number of persistence intervals corresponding to $k$-dimensional homological features and containing $r$. So, the persistence diagram can be seen as a multiscale topological signature encoding the homology of the union of balls for all radii as well as its evolution across the values of $r$. 

\begin{figure} 
	\centering
		\includegraphics[width = 1 \columnwidth]{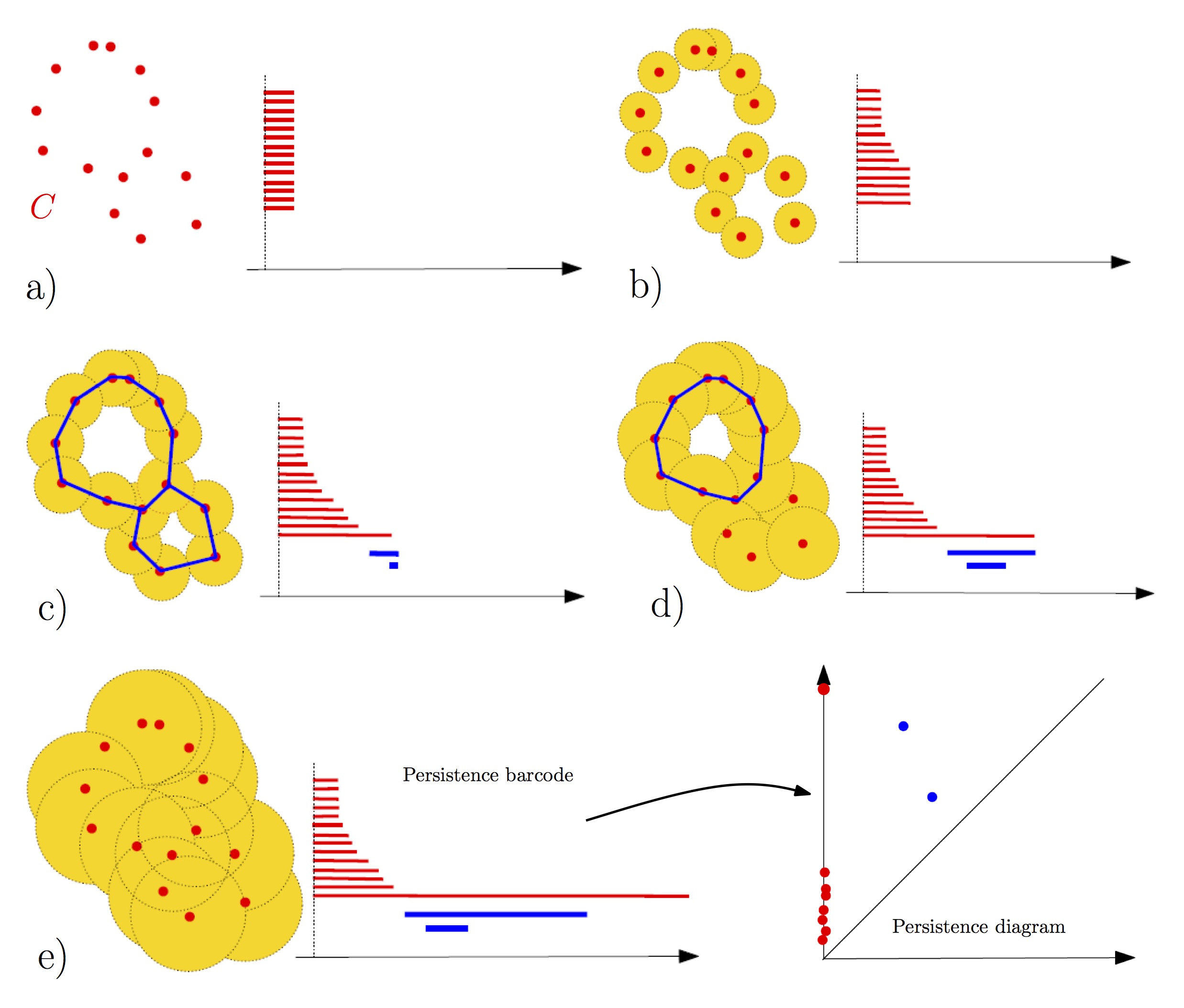}
		\caption{The sublevel set filtration of the distance function to a point cloud and the ``construction'' of its persistence barcode as the radius of balls increases. The blue curves in the unions of balls represent $1$-cycles associated to the blue bars in the barcodes. 
%a): For the radius $r=0$, the union of balls is reduced to the initial finite set of point, each of them corresponding to a $0$-dimensional feature, i.e. a connected component; the persistence diagrams creates an interval for the {\em birth} at $r=0$ for each of these features. b): some of the balls started to overlap resulting in the {\em death} of some connected components that get merged into others; the persistence diagram keeps track of these deaths, putting an end point to the corresponding intervals as they disappear. c): new components have merged giving rise to a single connected component and all the intervals associated to a $0$-dimensional feature have been ended, except the one corresponding to the remaining components; two new $1$-dimensional features, have appeared resulting in two new intervals (in blue) starting at their birth. d): one of the two 1-dimensional cycles has been filled, resulting in its death in the filtration and the end of the corresponding blue interval. e): all the $1$-dimensional features have died, it only remains the long (and never dying) red interval. The final barcode can also be equivalently represented as a persistence diagram where every interval $(a,b)$ is represented by the the point of coordinate $(a,b)$ in $\R^2$.
%\fred{Expliquer les cycles bleus dans la figure qui correspondent aux barres bleues du dgm, pour eviter la confusion avec les complexes simpliciaux }
}
	\label{fig:Cech-diag-example}
\end{figure} 

\subsection{Persistent modules and persistence diagrams}

%\fred{definir pers betti numbers}

Persistent diagrams can be formally and rigorously defined in a purely algebraic way. This requires some care and we only give here the basic necessary notions, leaving aside technical subtelties and difficulties. We refer the readers interested in a detailed exposition to \cite{chazal2012structure}.  

Let $\Filt = ( F_r )_{r \in T}$ be a filtration of a simplicial complex or a topological space. Given a non negative  integer $k$ and considering the homology groups $H_k(F_r)$ we obtain a sequence of vector spaces where the inclusions $F_r \subset F_{r'}$, $r \leq r'$ induce linear maps between $H_k(F_r)$ and $H_k(F_{r'})$. Such a sequence of vector spaces together with the linear maps connecting them is called a {\em persistence module}. 

\begin{definition}
A {persistence module} $\mathbb{V}$ over a subset $T$ of the real numbers $\R$ is an indexed family of vector spaces
$
(V_r \mid r \in T)
$
and a doubly-indexed family of linear maps 
$
(v_s^r: V_r \to V_s \mid r \leq s)
$
which satisfy the composition law
$
v_t^s \circ v_s^r = v_t^r
$
whenever $r \leq s \leq t$, and where $v_r^r$ is the identity map on $V_r$.
\end{definition}

In many cases, a persistence module can be decomposed into a direct sum of {\em intervals} modules $\mathbb{I}_{(b,d)}$ of the form 
$$
\cdots \rightarrow 0 \rightarrow \cdots \rightarrow 0 \rightarrow \Z_2 \rightarrow \cdots \rightarrow \Z_2 \rightarrow 0 \rightarrow \cdots
$$
where the maps $\Z_2 \to \Z_2$ are identity maps while all the other maps are $0$. Denoting $b$ (resp. $d$) the infimum (resp. supremum) of the interval of indices corresponding to non zero vector spaces, such a module can be interpreted as a feature that appears in the filtration at index $b$ and disappear at index $d$. 
When a persistence module $\mathbb{V}$ can be decomposed as a direct sum of interval modules, one can show that this decomposition is unique up to reordering the intervals (see \cite[Theorem 2.7]{chazal2012structure}). As a consequence, the set of resulting intervals is independent of the decomposition of $\mathbb{V}$ and is called the {\em persistence barcode} of $\mathbb{V}$. As in the examples of the previous section, each interval $(b,d)$ in the barcode can be represented as the point of coordinates $(b,d)$ in the plane $\R^2$.  The disjoint union of these points, together with the diagonale $\Delta = \{x = y \}$ is a multi-set called the the {\em persistence diagram} of $\mathbb{V}$.

The following result, from \cite[Theorem 2.8]{chazal2012structure}, gives some necessary conditions for a persistence module to be decomposable as a direct sum of interval modules. 

\begin{theorem} 
Let $\mathbb{V}$ be a persistence module indexed by $T \subset \R$. If $T$ is a finite set or if all the vector spaces $V_r$ are finite-dimensional, then $\mathbb{V}$ is decomposable as a direct sum of interval modules. 
Moreover, for any $s, t \in T$, $ s \leq t$, the number $\beta_{t}^s$ of intervals starting before $s$ and ending after $t$ is equal to the rank of the linear map $v_t^s$ and is called the $(s,t)$-persistent Betti number of the filtration.   
\end{theorem}

As both conditions above are satisfied for the persistent homology of filtrations of finite simplicial complexes, an immediate consequence of this result is that the persistence diagrams of such filtrations are always well-defined.

Indeed, it is possible to show that persistence diagrams can be defined as soon as the following simple condition is satisfied. 

\begin{definition}
A persistence module $\mathbb{V}$ indexed by $T \subset \R$ is {\em q-tame} if for any $r < s$ in $T$, the rank of the linear map $v_s^r : V_r \to V_s$ is finite.
\end{definition}

\begin{theorem}[\cite{ccggo-ppmtd-09,chazal2012structure}]
If $\mathbb{V}$ is a q-tame persistence module, then it has a well-defined persistence diagram. 
Such a persistence diagram $\dgm(\V)$ is the union of the points of the diagonal $\Delta$ of $\R^2$, counted with infinite multiplicity, and a multi-set above the diagonal in $\R^2$ that is locally finite. Here, by locally finite we mean that for any rectangle $R$ with sides parallel to the coordinate axes that does not intersect $\Delta$, the number of points of $\dgm(\V)$, counted with multiplicity, contained in $R$ is finite. Also, the part of the diagram made of the points with infinite second coordinate is called the essential part of the diagram.
\end{theorem}

The construction of persistence diagrams of q-tame modules is beyond the scope of this paper but it gives rise to the same notion as in the case of decomposable modules. It can be done either by following the algebraic approach based upon the decomposability properties of modules, or by adopting a measure theoretic approach that allows to define diagrams as integer valued measures on a space of rectangles in the plane. We refer the reader to \cite{chazal2012structure} for more informations.

Although persistence modules encountered in practice are decomposable, the general framework of q-tame persistence module plays a fundamental role in the mathematical and statistical analysis of persistent homology. In particular, it is needed to ensure the existence of limit diagrams when convergence properties are studied - see Section \ref{sec:statAspectsPH}.  

A filtration $\Filt = ( F_r )_{r \in T}$ of a simplicial complex or of a topological space is said to be tame if for any integer $k$, the persistence module $(H_k(F_r) \mid r \in T)$ is $q$-tame. Notice that the filtrations of finite simplicial complexes are always tame. As a consequence, 
for any integer $k$ a persistence diagram denoted $\dgm_k( \Filt)$ is associated to the filtration $\Filt$.
When $k$ is not explicitly specified and when there is no ambiguity, it is usual to drop the index $k$ in the notation and to talk about ``the'' persistence diagram $\dgm( \Filt)$ of the filtration $\Filt$. This notation has to be understood as ``$\dgm_k(\Filt)$ for some $k$''. 

%%%%%%%%%%%%%%%%%%%%%%%%%%%%%%%%%%%%%%%%%%%%%%%%%%
%%%%%%%%%%%%%%%%%%%%%%%%%%%%%%%%%%%%%%%%%%%%%%%%%%
\subsection{Persistence landscapes} \label{sec:land}
%%%%%%%%%%%%%%%%%%%%%%%%%%%%%%%%%%%%%%%%%%%%%%%%%%
%%%%%%%%%%%%%%%%%%%%%%%%%%%%%%%%%%%%%%%%%%%%%%%%%%

The persistence landscape introduced in \cite{bubenik2015statistical} 
 is an alternative representation of persistence diagrams. This approach aims at representing the topological information encoded in persistence diagrams as elements of
an Hilbert space, for which statistical learning methods can be directly applied. The persistence landscape is a collection of continuous, piecewise linear functions~\mbox{$\lscape \colon  \mathbb{N} \times \R \to \R$} that summarizes a
persistence diagram $\dgm$. 
\begin{figure}[h]
	\centering
		\includegraphics[width = 0.95 \columnwidth]{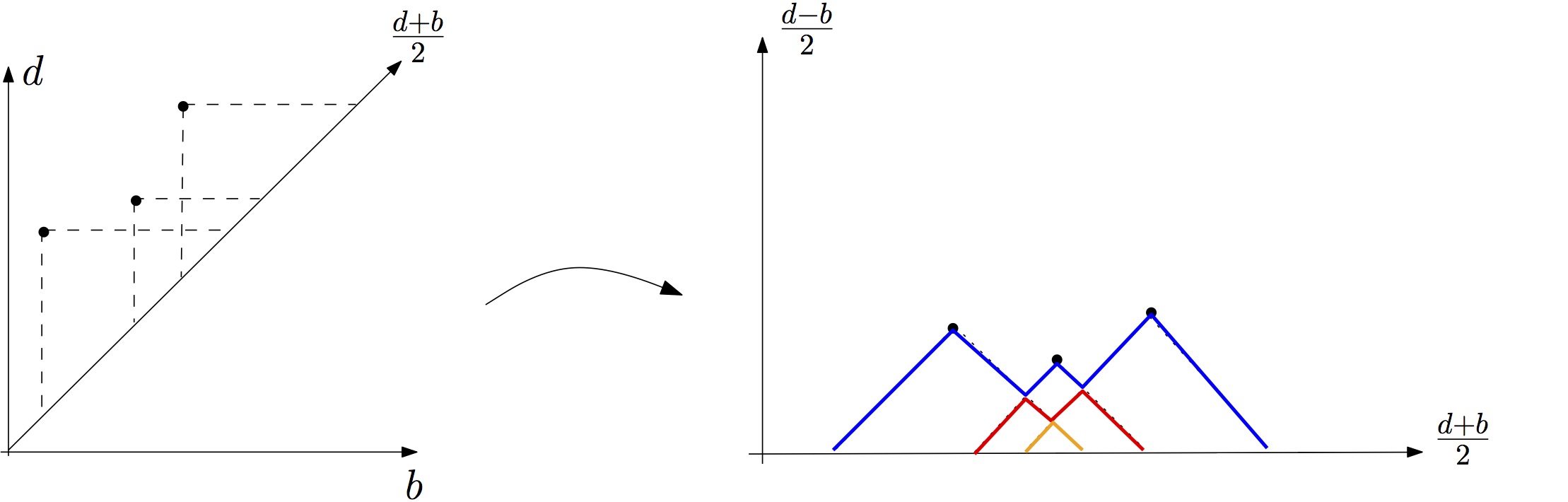} 
		\caption{An example of persistence landscape (right) associated to a persistence diagram (left). The first landscape is in blue, the second one in red and the last one in orange. All the other landscapes are zero.}
	\label{fig:PersistenceLandscapeExample}
\end{figure} 

A birth-death pair $p=(b,d) \in \dgm$  is transformed into the point $ \left( \frac{ b + d}{2}, \frac{d - b}{2}  \right)$, see Figure~\ref{fig:PersistenceLandscapeExample}.  Remember that the  points with infinite persistence have been simply discarded in this definition.
The landscape is then defined by considering the set of functions created by tenting the features  of  the rotated persistence diagram as follows:
\begin{align}
 \Lambda_p (t) &=
% \begin{cases}
%  t-x+y & t \in [x-y, x] \\
%  x+y-t & t \in (x,  x+y] \\
%  0 & \text{otherwise}
% \end{cases} \nonumber \\
%    & =
 \begin{cases}
  t-b& t \in [b, \frac{b+d}{2}] \\
  d-t & t \in (\frac{b+d}{2},d] \\
  0 & \text{otherwise}.
 \end{cases}
 \label{eq:triangle}
\end{align}

The persistence landscape $\lscape_{\dgm}$ of $\dgm$ is a  summary of the arrangement of piecewise linear curves obtained by overlaying the graphs of
the functions~\mbox{$\{ \Lambda_p \}_{p \in \dgm}$}.
% To avoid  minor  technical difficulties, we restrict our attention to persistence landscapes for metric spaces $\X$ such that $( \alpha_\textrm{birth}, \alpha_\textrm{death}) \in [0,T] \times [0,T]$ for all $(\alpha_\textrm{birth},\alpha_\textrm{death}) \in \dgm (\Filt(\X))$, for some fixed $T>0$\footnote{The point $(0,\infty)$, from zero persistence, is also removed.}. 
Formally, the persistence landscape of $\dgm$ is the collection of functions
\begin{equation}\label{eq:landscape}
 \lscape_{\dgm}(k,t) = \underset{r \in \dgm}{\text{kmax}} ~ \Lambda_r(t), \quad
 t \in [0,T], k \in \mathbb{N},
\end{equation}
where  kmax  is the $k$th largest value in the set; in particular,
$1$max is the usual maximum~function. Given $k \in \mathbb{N}$, the function $\lscape_{\dgm}(k,.) : \R \to \R$ is called the \emph{$k$-th landscape} of $\dgm$.
It is not difficult to see that the map that associate to each persistence diagram its corresponding landscape is injective. In other words, formally no information is lost when a persistence diagram is represented through its persistence landscape. 

The advantage of the persistence landscape representation is two-fold. First, persistence diagrams are mapped as elements of a functional space, opening the door to the use of a broad variety of statistical and data analysis tools for further processing of topological features - see, e.g. \cite{bubenik2015statistical,chazal2015stochastic} and Section \ref{sec:submeth}. Second, and fundamental from a theoretical perspective, the persistence landscapes share the same stability properties as persistence diagrams - see Section \ref{subsec:stab}.

%%%%%%%%%%%%%%%%%%%%%%%%%%%%%%%%%%%%%%%%%%%%%%%%%%
%%%%%%%%%%%%%%%%%%%%%%%%%%%%%%%%%%%%%%%%%%%%%%%%%%

%%%%%%%%%%%%%%%%%%%%%%%%%%%%%%%%%%%%%%%%%%%%%%%%%%
%%%%%%%%%%%%%%%%%%%%%%%%%%%%%%%%%%%%%%%%%%%%%%%%%%
\subsection{Linear representations of persistence homology}
%%%%%%%%%%%%%%%%%%%%%%%%%%%%%%%%%%%%%%%%%%%%%%%%%%
%%%%%%%%%%%%%%%%%%%%%%%%%%%%%%%%%%%%%%%%%%%%%%%%%%

A persistence diagram without its essential part  can be represented as a discrete measure on $\Delta^+ =  \{  p = (b,d),\ b < d < \infty\}$. With a slight abuse of notation, we can write :
\begin{equation*}
\dgm = \sum_{ p\in \dgm} \delta_{p},
\end{equation*} 
where the features are counted with multiplicity and where $\delta_{(b,d)}$ denotes the Dirac measure in $p = (b,d)$.
% and where  points with infinite persistence are simply discarded. 
Most of the persistence-based descriptors that have been proposed to analyze persistence can be expressed as linear transformations of  the persistence diagram, seen as a point process 
$$\Psi(\dgm) =   \sum_{p  \in \dgm} f(p),$$
for some function $f$ defined on $\Delta$ and taking values in a Banach space.

In most cases, we want these transformations to apply independently at each homological dimension. For  $k \in \N$ a given homological dimension we then consider some linear transformation of the persistence diagram restricted to the topological features of dimension $k$:
\begin{equation}
\Psi_k(\dgm_k) =   \sum_{p \in \dgm_k} f_k(p), 
\label{eq:linear_repr}
\end{equation} 
where $\dgm_k$ is the persistence diagram of the topological features of dimension $k$ and where $f_k$ is  defined on $\Delta$ and takes values in a Banach space.

\paragraph{Betti curve.} The simplest way to represent persistence homology is the Betti function or Betti curve. The Betti curve of homological dimension $k$ is defined as
\begin{equation*}
\beta_k(t) =  \sum_{(b,d) \in \dgm} w(b,d)  \indic{ t \in [b,d]} 
\end{equation*}  
where $w$ is a weight function defined on $\Delta$. In other words, the Betti curve is the number of barcodes at time $m$. This descriptor is a linear representation of  persistence homology by taking $f$ in \eqref{eq:linear_repr} such that  $ f(b,d)(t)=  w(b,d)  \indic{ t \in [b,d]} .$ 
A typical choice for the weigh function is a increasing function of the persistence $ w(b,d)  = \tilde w(d-b)$ where $\tilde w$ is an increasing function defined on $\R^+$.  One of the first applications of Betti curves can be found in \cite{umeda2017time}.

\paragraph{Persistence surface.} Persistence surface (also called persistence images)   is obtained  by making the convolution of a diagram with a kernel. It has been introduced  in \cite{adams2017persistence}. For $K:\R^2\to \R$ a kernel and $H$ a $2\times 2$ bandwidth matrix (e.g. a symmetric positive definite matrix), let for $u\in \R^2$ 
\begin{equation*}
K_H(u) = \det(H)^{-1/2} K(H^{-1/2}  u).
\end{equation*}
Let $w:\R^2 \to \R_+$ a weight function defined on  $\Delta$. One defines the persistence surface of homological dimension $k$ associated to a diagram $ \dgm$, with kernel $K$ and  bandwidth matrix $H$ by:
\begin{equation*}
\forall u \in \R^2, \ \rho_k(\dgm)(u) = \sum_{p \in \dgm_k} w(r) K_H(u-p) .
\end{equation*}
The persistence surface is obviously a linear representation of persistence homology.  Typical weigh functions are increasing functions of the persistence. 

\paragraph{Other linear representations of persistence.} Many other linear representations of persistence have been proposed in the literature, such as e. g., the persistence silhouette~\citep{chazal2015stochastic}, the accumulated persistence function~ \citep{biscio2019accumulated} and variants of the persistence surface~\citep{chen2017density, kusano2016persistence, reininghaus2015stable}.   

Considering persistence diagrams as discrete measure and their vectorizations as linear representation is an approach that also proven fruitful to study distributions of diagrams \cite{Divol_Chazal_2020} and the metric structure of the space of persistence diagrams \cite{divol2019understanding} - see Sections \ref{sec:metric-space-dgm} and \ref{sec:stat-family}.

\subsection{Metrics on the space of persistence diagrams} \label{sec:metric-space-dgm}
To exploit the topological information and topological features inferred from persistent homology, one needs to be able to compare persistence diagrams, i.e. to endow the space of persistence diagrams with a metric structure. Although several metrics can be considered, the most fundamental one is known as the {\em bottleneck distance}. 

Recall that a persistence diagram is the union of a discrete multi-set in the half-plane above the diagonal $\Delta$ and, for technical reasons that will become clear below, of $\Delta$ where the point of $\Delta$ are counted with infinite multiplicity. 
A {\em matching} - see Figure \ref{fig:MatchingDiag} - between two diagrams $\dgm_1$ and $\dgm_2$ is a subset $m \subseteq \dgm_1 \times \dgm_2$ such that every points in $\dgm_1 \setminus \Delta$ and $\dgm_2 \setminus \Delta$ appears exactly once in $m$. In other words, for any $p \in \dgm_1 \setminus \Delta$, and for any $q \in \dgm_2 \setminus \Delta$, $(\{ p \} \times \dgm_2) \cap m$ and $(\dgm_1 \times \{ q \}) \cap m$ each contains a  single pair. 
The {\em Bottleneck distance} between $\dgm_1$ and $\dgm_2$ is then defined by
$$
\bottle (\dgm_1, \dgm_2) =\inf_{\scriptsize{\mbox{matching}} \ m} \max_{(p,q) \in m} \| p-q \|_\infty.
$$

\begin{figure}[h]
	\centering
		\includegraphics[width = 0.5 \columnwidth]{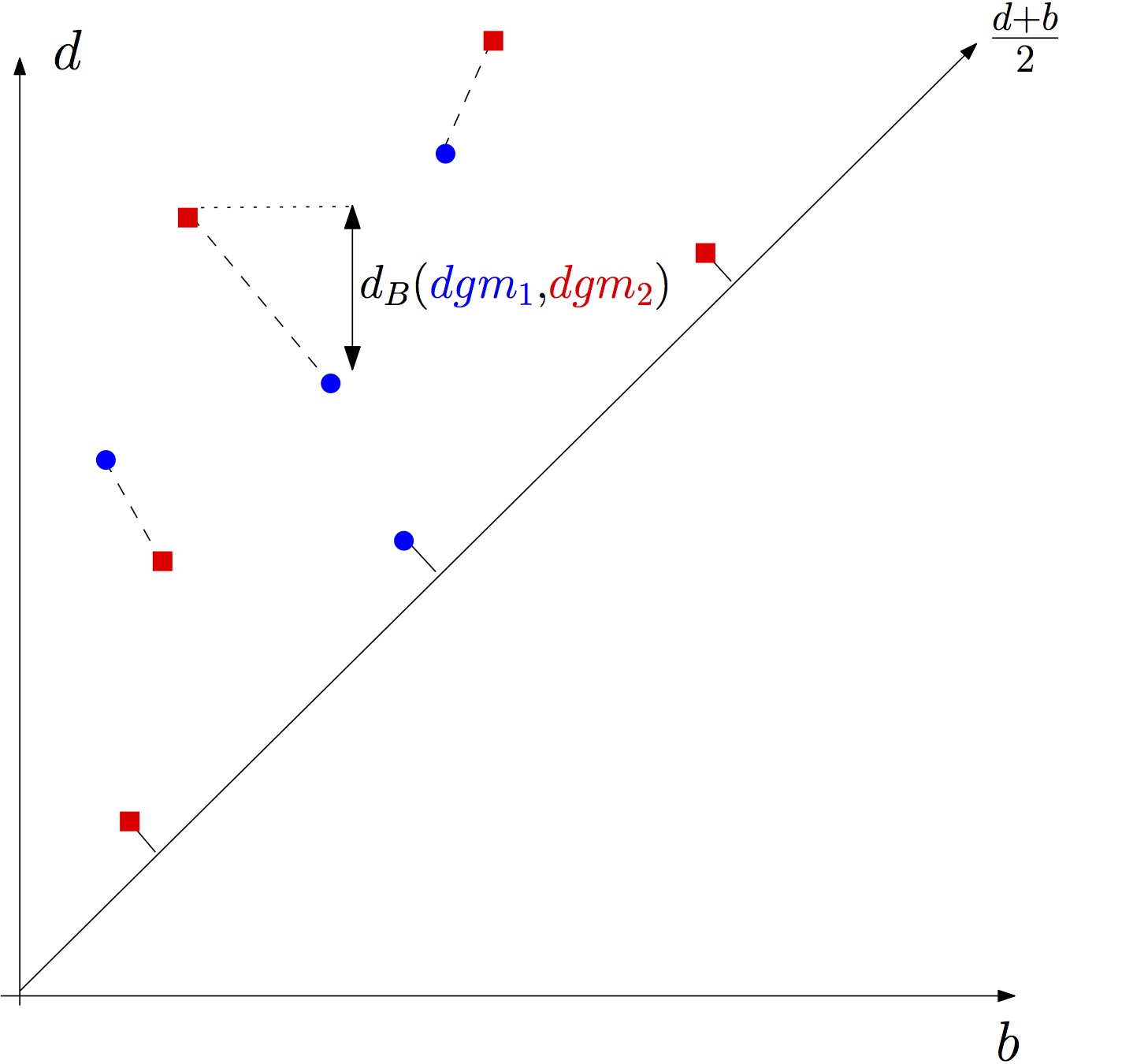} 
		\caption{A perfect matching and the Bottleneck distance between a blue and a red diagram. Notice that some points of both diagrams are matched to points of the diagonal.}
	\label{fig:MatchingDiag}
\end{figure}

The practical computation of the bottleneck distance boils down to the computation of a perfect matching in a bipartite graph for which classical algorithms can be used. 

The bottleneck metric is a $L_\infty$-like metric. It turns out to be the natural one to express stability properties of persistence diagrams presented in Section \ref{subsec:stab}, but it suffers from the same drawbacks as the usual  $L_\infty$ norms, i.e. it is completely determined by the largest distance among the pairs and do not take into account the closeness of the remaining pairs of points. A variant, to overcome this issue, the so-called Wasserstein distance between diagrams is sometimes considered. Given $p \geq 1$, it is defined by 
$$
W_p (\dgm_1, \dgm_2)^p = \inf_{\scriptsize{\mbox{matching}} \ m} \sum_{(p,q) \in m} \| p-q \|^p_\infty.
$$
Useful stability results for persistence in the $W_p$ metric exist among the literature, in particular \cite{cohen2010lipschitz}, but they rely on assumptions that make them consequences of the stability results in the bottleneck metric. 
A general study of the space of persistence diagrams endowed with $W_p$ metrics has been considered in \cite{divol2019understanding} where they propose a general framework, based upon optimal partial transport, in which many important properties of persistence diagrams can be proven in a natural way.

\subsection{Stability properties of persistence diagrams} \label{subsec:stab}

A fundamental property of persistence homology is that persistence diagrams of filtrations built on top of data sets turn out to be very stable with respect to some perturbations of the data. To formalize and quantify such stability properties, we first need to precise the notion of perturbation that are allowed. 
%and to define a metric on the space of persistence diagrams.  

Rather than working directly with filtrations built on top of data sets, it turns out to be more convenient to define a notion of proximity between persistence module from which we will derive a general stability result for persistent homology. Then, most of the stability results for specific filtrations will appear as a consequence of this general theorem. To avoid technical discussions, from now on we assume, without loss of generality, that the considered persistence modules are indexed by $\R$. 

\begin{definition}
Let $\V, \W$ be two persistence modules indexed by $\R$. Given $\delta \in \R$, a homomorphism of degree $\delta$ between $\V$ and $\W$ is a collection $\Phi$ of linear maps $\phi_r : V_r \to W_{r+\delta}$, for all $r \in \R$ such that for any $r \leq s$, $\phi_s \circ v_s^r = w_{s+\delta}^{r+\delta} \circ \phi_r$.
\end{definition}

An important example of homomorphism of degree $\delta$ is the {\em shift endomorphism} $1_\V^\delta$ which consists of the families of linear maps $(v_{r+\delta}^r)$. Notice also that homomorphisms of modules can naturally be composed: the composition of a homomorphism $\Psi$ of degree $\delta$ between $\mathbb{U}$ and $\V$ and a homomorphism $\Phi$ of degree $\delta'$ between $\V$ and $\W$ naturally gives rise to a homomorphism $\Phi \Psi$ of degree $\delta + \delta'$ between $\mathbb{U}$ and $\W$.

\begin{definition}
Let $\delta \geq 0$. Two persistence modules $\V, \W$ are $\delta$-interleaved if there exists two homomorphism of degree $\delta$, $\Phi$, from $\V$ to $\W$ and $\Psi$, from $\W$ to $\V$ such that $ \Psi \Phi = 1_\V^{2\delta}$ and $ \Phi \Psi = 1_\W^{2\delta}$.
\end{definition}

Although it does not define a metric on the space of persistence modules, the notion of closeness between two persistence module may be defined as the smallest non negative $\delta$ such that they are $\delta$-interleaved. Moreover, it allows to formalize the following fundamental theorem \cite{ccggo-ppmtd-09,chazal2012structure}. 

\begin{theorem}[Stability of persistence] \label{thm:pers-stab-alg}
Let $\V$ and $\W$ be two q-tame persistence modules. If $\V$ and $\W$ are $\delta$-interleaved for some $\delta \geq 0$, then
$$\bottle ( \dgm(\V), \dgm(\W) ) \leq \delta.$$
\end{theorem}

Although purely algebraic and rather abstract, this result is a efficient tool to easily establish concrete stability results in TDA.  
For example we can easily recover the first persistence stability result that appeared in the literature \citep{steiner05stable}. 

\begin{theorem} \label{thm:pers-stab-func}
Let $f, g: M \to \R$ be two real-valued functions defined on a topological space $M$ that are $q$-tame, i.e. such that the sublevel sets filtrations of $f$ and $g$ induce $q$-tame modules at the homology level. Then for any integer $k$,  
$$ \bottle( \dgm_k(f), \dgm_k(g) )  \leq \| f-g \|_\infty = \sup_{x \in M} | f(x)-g(x)|$$
where $\dgm_k(f)$ (resp. $\dgm_k(g)$) is the persistence diagram of the persistence module $( H_k(f^{-1}(-\infty,r])) | r \in \R)$ (resp. $( H_k(g^{-1}(-\infty,r])) | r \in \R)$) where the linear maps are the one induced by the canonical inclusion maps between sublevel sets. 
\end{theorem}

\begin{proof}
Denoting $\delta = \| f-g \|_\infty$ we have that for any $r \in \R$, $f^{-1}(-\infty,r]) \subseteq g^{-1}(-\infty,r+\delta])$ and $g^{-1}(-\infty,r]) \subseteq f^{-1}(-\infty,r+\delta])$. This interleaving between the sublevel sets of $f$ induces a $\delta$-interleaving between the persistence modules at the homology level and the result follows from the direct application of Theorem \ref{thm:pers-stab-alg}. 
\end{proof}

Theorem \ref{thm:pers-stab-alg} also implies a stability result for the persistence diagrams of filtrations built on top of data. 

\begin{theorem} \label{thm:pers-stab-data}
Let $\X$ and $\Y$ be two compact metric spaces and let $\Filt(\X)$ and $\Filt(\Y)$ be the Vietoris-Rips of \v Cech filtrations built on top $\X$ and $\Y$. 
Then 
$$ d_b \left ( \dgm ( \Filt( \X) ) , \dgm ( \Filt( \Y) ) \right ) \leq 2 d_{GH}(\X, \Y) $$
where $\dgm ( \Filt( \X) )$ and $\dgm ( \Filt( \Y) )$ denote the persistence diagram of the filtrations $\Filt(\X)$ and $\Filt(\X)$.
\end{theorem}

As we already noticed in the Example 3 of Section \ref{sec:pers-example}, the persistence diagrams can be interpreted as multiscale topological features of $\X$ and $\Y$. In addition, Theorem \ref{thm:pers-stab-data} tells us that these features are robust with respect to perturbations of the data in the Gromov-Hausdorff metric. They can be used as discriminative features for classification or other tasks - see, for example, \cite{ccgmo-ghsssp-09} for an application to non rigid 3D shapes classification. 

We now give similar results for the alternative persistence homology representations introduced before.  From the definition of persistence landscape  we immediately observe that
$\lscape (k,\cdot)$ is one-Lipschitz and thus similar stability properties are satisfied for the landscapes as for persistence diagrams.
\begin{Prop}[Stability of persistence landscapes  \cite{bubenik2015statistical}. ]
\label{lemma:basic-prop}
 Let $\dgm$ and $\dgm'$ be two persistence diagrams (without their essential parts). For any $t\in
\mathbb{R}$ and any $k \in \mathbb{N}$, we have:\\
(i)  $\lambda(k,t)
~\geq~ \lambda(k+1,t) ~\geq~ 0 $.\\
(ii) $| \lambda (k,t) -\lambda'(k,t)| ~\leq~  \bottle(\dgm,\dgm') )$. 
\end{Prop}
A large class of linear representations are continuous with respect to the Wasserstein metric $W_s$  in the space of persistence diagrams and with respect to the Banach norm of the linear representation of persistence. Generally speaking, it is not always possible to upper bound the modulus of continuity of the linear representation operator. However, in the case where $s=1$ it is even possible to show a stability result if the weight function  take small values for points close to the diagonal, see  \cite{divol2019understanding,hofer2019learning}.

\paragraph{Stability versus discriminative capacity of persistence representations.}  The results of \cite{divol2019understanding} show that continuity and stability is only possible with weigh functions taking small values for points close to the diagonal. However, in general there is no specific reason to consider that points close to the diagonal are less important than others, given a learning task. In a machine learning perspective, it is also relevant to design linear representation with general weigh functions, although it would be more difficult to prove consistency of the corresponding methods without at least the continuity of the representation. Stability is thus important but maybe a too strong requirement for many problems in data sciences. Designing linear representation that are sensitive to specific parts of a persistence diagrams rather than globally stable may reveal a good strategy in practice.    
%A related question is : can we have  stability with the opposite inequality ? Do close representations imply close supports or close sample measures of the data points ? However that and it is not possible to state such results in general. 
%\bert{est ce que cela correspond aux questions du type probleme inverse TDA ?}

\section{Statistical aspects of persistent homology } \label{sec:statAspectsPH}

Persistence homology by itself does not take into account the random nature of data and the intrinsic variability of the topological quantity they infer. We now present a statistical approach to persistent homology, in the sense that data is considered as generated from an unknown distribution. We start with several  consistency results for persistent homology inference.

 \subsection{Consistency results for persistent homology}
Assume that we observe $n$
points $(X_1, \dots ,X_n)$  in a metric space $(M,\rho)$ drawn i.i.d.\ from an unknown probability measure $\mu$ whose support is a compact set denoted $\X_\mu$.
The Gromov-Hausdorff distance allows us to compare $\X_\mu$ with compact metric spaces 
not necessarily embedded in  $M$. 
%We thus consider $(\X_\mu,\rho) $ as an element of $\mathcal K$ (rather than an element of $\mathcal K (M)$).
 In the following, an \emph{estimator} $\hX$ of $\X_\mu$ is  a function of $X_1 \dots, X_n$ that takes values in the set of compact metric spaces. 
%and which is measurable for the Borel algebra induced by $\dgh$. \bert{eviter cette intro un peu pompeuse ?}

Let $\Filt(\X_\mu)$ and $\Filt(\hX)$ be two filtrations defined on $\X_\mu$ and $\hX$.  Starting from Theorem~\ref{thm:pers-stab-data}; an natural strategy for estimating the persistent homology of $\Filt(\X_\mu)$ consists in estimating the support $\X_\mu$. Note that in some cases the space $M$ can be unknown and the observations $X_1 \dots, X_n$ are then only known through their pairwise
distances $\rho(X_i,X_j)$, $i,j = 1, \cdots, n$. The use of the Gromov-Hausdorff distance  allows us to consider this set of observations as an
abstract metric space of cardinality $n$, independently of the way it is embedded in $M$.  
This general framework includes the more standard approach consisting in estimating the support with respect to the Hausdorff distance by restraining the values of  $\hX$ to the compact sets included in $M$. 
%Using equation (\ref{eq:dgmdgh}), the problem of persistence diagrams estimation reduces to the better known problem of estimating the support of a measure.

The finite set $\X _n := \{ X_1, \dots ,X_n\}$ is a natural estimator of the support $\X_\mu$.
In several contexts discussed in the following,
$\X _n$ shows optimal rates of convergence to $\X_\mu$ with respect to the Hausdorff distance.
For some constants $a,b >0$, we say that $\mu$ satisfies the $(a,b)$-standard assumption if for any $x \in \X_\mu$ and any 
$r >0$, 
\begin{equation}
\label{abstAssump}
\mu(B(x,r)) \geq \min(ar^b, 1). 
\end{equation}
This assumption has been
widely used in the literature of set estimation under Hausdorff distance \citep{cuevas2004boundary,SinghScottNowak09}. Under this assumption, it can be easily derived that the rate of convergence of $ \dgm(\Filt(\X_n)) $ to $\dgm (\Filt(\X_\mu))$ for the bottleneck metric is upper bounded by $O\left( \frac{ \log n}  n  \right)^{1/b}$. More precisely, this rate upper bounds the minimax rate of convergence over the set of
probability measures on the metric space $(M,\rho)$ satisfying the $(a,b)$-standard assumption on $M$.
\begin{theorem}{\textnormal{[\cite{chazal2014convergenceJMLR}]}} \label{prop:lowerboundAbs}
For some positive constants $ a $ and $ b $, let
 \begin{equation*} \label{Pad}
   \mathcal P :=   \left\{ \mu  \textrm{ on } M \: | \:  \X_\mu \textrm{ is compact and } \forall x \in \X_\mu,   
    \forall r > 0 , \, \mu \left(B(x,r) \right) \geq  \min \left( 1 , a r^b \right) \right\}  .
\end{equation*}
Then, it holds
\begin{equation*} \label{ubE}
 \sup_{\mu \in \mathcal P} \E \left[ \bottle  (   \dgm (\Filt(\X_\mu)) ,   \dgm(\Filt(\X_n)) ) \right]  \leq C  \left( \frac{ \log n}  n  \right)^{1/b} 
\end{equation*}
where the constant $C$ only depends on  $a$ and $b$. 
\end{theorem}
Under additional technical assumptions, the corresponding lower bound  can be shown (up to a logarithmic term), see \cite{chazal2014convergenceJMLR}. By applying stability results, similar consistency results can be easily derived under alternative generative models as soon as a consistent estimator of the support under Hausdorff metric is known. For instance,  from the results of \citet{GenoveseEtAl2012} about Hausdorff support estimation under additive noise, it can be deduced  that the  minimax convergence rates for the persistence diagram estimation is faster than $( \log n) ^{-1/2}$. Moreover, as soon as a stability result is available for some given representation of persistence,  similar consistency results can be directly derived from the consistency for persistence diagrams.

\paragraph{Estimation of the persistent homology of functions.} Theorem~\ref{thm:pers-stab-alg} opens the door to the estimation of the persistent homology of functions defined on $\R^d$, on a submanifold of $\R^d$ or more generally on a metric space.  The persistent homology of regression functions has also been studied in \cite{bubenik2010statistical}. The alternative approach of~\citet{bobrowski2014topological} which is based on the inclusion map between nested pairs of estimated level sets can  be applied with  kernel density and regression kernel estimators to estimate persistence homology of density functions and regression functions. Another direction of research on this topic concerns various versions of robust TDA. One solution is to study the persistent homology  of the upper level sets of density estimators~\citep{fasy2014confidence}. A different approach, more closely related to the distance function, but robust to noise, consists in studying the persistent homology of the sub level sets of the distance to measure defined in Section~\ref{subsec:DTM} \citep{cflmrw-rtidt-14}.

\subsection{Statistic of persistent homology computed on a point cloud}

For many applications, in particular when the support of the point cloud is not drawn on or close to a geometric shape, persistence diagrams can be quite complex to analyze. In particular, many topological features are closed to the diagonal. Since they correspond to topological structures that die very soon after they appear in the filtration, these points are generally considered as noise, see Figure~\ref{fig:Protein4graphs} for an illustration. Confidence regions of persistence diagram are rigorous answers to the problem of distinguishing between signal and noise in these representations. 

The stability results given in Section~\ref{subsec:stab} motivate the use of the bottleneck distance to define confidence regions. However alternative distances in the spirit of Wasserstein distances can be proposed too. When estimating a persistence diagram  $\dgm$ with an estimator $\widehat  \dgm$, we typically look for some value $\eta_{\alpha}$ such that
\begin{equation*}
 \label{conf}
 P({\bottle(   \widehat  \dgm ,   \dgm)   \geq \eta_{\alpha} }) \leq  \alpha ,
 \end{equation*}
 for $\alpha \in (0,1)$.
Let $B_\alpha $ be the closed ball of radius $\alpha$ for the bottleneck distance and centered at $\widehat  \dgm $ in the space of persistence diagrams. 
Following~\cite{fasy2014confidence}, we can visualize the signatures of the points belonging to this ball in various ways. One first option is to center a box of side length $2\alpha$ at each point of the persistence diagram $\widehat  \dgm$. An alternative solution is to visualize the confidence set by adding a band at (vertical) distance $\eta_{\alpha}/2$ from the diagonal (the bottleneck distance being defined for the $\ell_\infty$ norm),  see Figure~\ref{fig:Protein4graphs} for an illustration. The points outside the band are then considered as significant topological features, see~\cite{fasy2014confidence} for more details.  
 
Several methods have been proposed in~\cite{fasy2014confidence} to estimate $ \eta_{\alpha}$ in different frameworks. These methods mainly rely from
stability results for persistence diagrams: confidence sets for diagrams can be derived from confidence sets in the sample space. 

\paragraph{Subsampling approach.} This method is based on a confidence region for the support $K$ of the distribution of the sample in Hausdorff distance. Let $\tilde{\X}_b$ be a subsample of size $b$ drawn from the sample $\tilde{\X}_n$ , where $b  = o(n / log n)$. Let  $q_b(1-\alpha)$ be the quantile  of of the distribution of $ \dhaus \left( \tilde{\X}_b, \X_n \right) $. Take $ \hat \eta_{\alpha}  := 2 \hat{q}_b(1-\alpha)$ where $\hat{q}_b$ is an estimation $q_b(1-\alpha)$ using  a standard Monte Carlo procedure. Under an $(a,b)$ standard assumption, and for $n$ large enough, \cite{fasy2014confidence}  show that 
\begin{eqnarray*}
P \left( \bottle \left( \dgm \left(  \Filt(K) \right)  ,  \dgm \left(  \Filt(\X_n) \right)  \right)  >
\hat \eta_\alpha   \right)
& \leq &
 P  \bigl(  \dhaus \left(K, \X _{n} \right)  > \hat \eta_\alpha \bigr) \\
&  \leq&  \alpha+ O \biggl( \frac{b }{n} \biggr)^{1/4}.
\end{eqnarray*}

\paragraph{Bottleneck Bootstrap.} The stability results often leads to conservative confidence sets. An alternative strategy is the bottleneck bootstrap introduced in \cite{cmm-rcrgi-15}.  We consider the general setting where a  persistence diagram $\widehat  \dgm$ is defined from the observation 
$(X_1, \ldots,  X_n)$ in a metric space. This persistence diagram corresponds to the estimation of an underlying persistence diagram $\dgm$, which can be related for instance to the support of the measure, or to the sublevel sets of a function related to this distribution (for instance a density function when the $X_i$'s are in $\R^d$). Let $(X_1^*,\dots,X_n^*)$ be a sample from the empirical measure defined from the observations $(X_1, \ldots,  X_n)$. Let also $\widehat  \dgm^*$ be the persistence diagram derived from this sample. We then can take for  $\eta_\alpha$    the quantity  $\hat \eta_\alpha$ defined by
\begin{equation}
P( \bottle ( \widehat  \dgm^*, \widehat  \dgm) > \hat \eta_\alpha \, |\, X_1,\ldots, X_n) = \alpha.
\end{equation}
Note that $\hat \eta_\alpha$ can be easily estimated with Monte Carlo procedures. 
It has been shown in~\cite{cmm-rcrgi-15} that the bottleneck bootstrap is valid when computing 
the sublevel sets of a density estimator. 

\paragraph{Bootstrapping Persistent Betti numbers.} As already mentioned,  confidence regions based on stability properties of persistence   may lead  to very conservative confidence regions. Based on the concepts of stabilizing statistics~\cite{penrose2001central},   asymptotic normality for persistent Betti numbers  has been shown recently by~\cite{krebs2019asymptotic,roycraft2020bootstrapping}, under very mild conditions on the filtration and the distribution of sample cloud. In addition, bootstrap procedures are also shown to be valid in this framework. More precisely, a smoothed bootstrap procedure together with a convenient rescaling of the point cloud seems to be a promising approach for boostrapping TDA features from point cloud data.

\subsection{Statistic for a family of persistent diagrams or other representations } \label{sec:stat-family}
%%%%%%%%%%%%%%%%%%%%%%%%%%%%%%%%%%%%%%%%%%%%%%%%%%%%

Up to now in this section, we were only considering statistics based on one single observed  persistence diagram. We now consider a new framework where several persistence diagrams (or other representations) are available and we are interested in providing central tendency, confidence regions and hypothesis tests for topological descriptors built on this family.

\subsubsection{Central tendency for persistent homology}
\label{sec:submeth}
%%%%%%%%%%%%%%%%%%%%%%%%%%%%%%%%%%%%%%%%%%%%%%%%%%%%

%In the previous section, we were only considering one persistence diagram and the confidence regions have been derived from stability and from the sampling inside the point cloud for which persistence homology has been computed. We now consider a new framework where several persistence diagrams (or other representations) are available and we are interested in providing central tendency descriptors for this family. 

\paragraph{Mean and expectations of distributions of diagrams.}
The space of persistence diagrams being a general metric space but not an Hilbert space, the definition of a \emph{mean persistence diagram} is not obvious and unique. One first natural approach to define a central tendency in this context is to consider Fr\'echet means of distributions of diagrams. Their existence has been proved in \cite{mileyko2011probability} and they have also been characterized in \cite{turner2014frechet}. However they are may not be unique and they turn out to be difficult to compute in practice. 
To partly overcome these problems, different approaches have been recently proposed based on numerical optimal transport \cite{lacombe2018large} or linear representations and kernel based methods \cite{Divol_Chazal_2020}.  

\paragraph{Topological signatures from subsamples}
Central tendency properties of persistent homology can also be used  to compute topological signatures for very large data sets, as an alternative approach to overcome the prohibitive cost of persistence computations. 
%To overcome the problem of computational costs,  sampling strategies can be proposed to compute topological signatures based on persistence landscapes. 
Given a large point cloud, the idea is to extract many subsamples, to compute the persistence landscape for each subsample
and then to combine the information. 
 
For any positive integer $m$, let  $\Xm = \{ x_1,\cdots , x_m \}
 $ be a sample of $m$ points drawn from a measure $\mu$ in a metric space $M$ and which support is denoted $\mathbb{X}_\mu$. We assume that the diameter of  $\mathbb{X}_\mu$ is finite and upper bounded by $\frac T2$, where $T$ is  the same constant as in the definition of persistence landscapes in Section~\ref{sec:land}. For ease of exposition, we focus on the case $k=1$,
and set $\lscape (t) = \lscape (1,t)$.  However,
the results we present in this section hold for~$k > 1$. The corresponding persistence landscape (associated to the persistence diagram of the \v Cech or Rips-Vietoris filtration) 
is $\lambda_{\Xm}$
and we denote by $\Psi_\mu^{m}$ the measure induced by $\mu^{\otimes
m}$ on the space of persistence landscapes.
Note that the persistence landscape $\lambda_{\Xm}$  can
be seen as a single draw from the measure $\Psi_\mu^{m}$.
The point-wise expectations of the (random) persistence landscape
under this
measure is defined by
$\mathbb{E}_{\Psi_\mu^{m}}[\lambda_{\Xm}(t)], t \in [0,T]$.
The average landscape $\mathbb{E}_{\Psi_\mu^{m}}[\lambda_{X}]$ has a natural
empirical counterpart, which can be used as
its unbiased estimator. Let $S_1^m, \dots, S_\ell^m$ be $\ell$ independent samples of
size $m$ from~$\mu^{\otimes m}$. We define the empirical average landscape as
\begin{equation}
\label{eq:avLand}
\overline{\lambda_\ell^m}(t)= \frac{1}{b} \sum_{i=1}^b \lambda_{S_i^m}(t), \quad
\text{ for all } t \in [0,T],
\end{equation}
and propose to  use
$\overline{\lambda_\ell^m}$ to estimate $\lambda_{{\mathbb{X}_\mu}}$.
Note that computing the persistent homology of $\X_n$ is $O(\exp(n))$, whereas computing
the average landscape is $O(b \exp(m))$.

Another motivation for this subsampling approach is that it can be also applied  when
$\mu$ is a discrete measure with support $\X_N=\{x_1, \dots, x_N\}
 $ lying in a metric space $M$. This framework  can be very common in practice,  when a continuous (but unknown measure) is approximated by a discrete uniform measure $\mu_N$ on $\X_N$.

%%%%%%%%%%%%%%%%%%%%%%%%%%%%%%%%%%%%%%%%%%
The average landscape $\mathbb{E}_{\Psi_\mu^{m}}[\lambda_{X}]$ is an interesting quantity on its own,
since it carries some stable topological information about the underlying
measure $\mu$, from which the data are generated.
 \begin{theorem}{\textnormal{[\cite{chazal2014subsampling}]}} \label{thm:stabW}
%Let $(\mathbb{M}, \rho)$ be a metric space of diameter bounded by $T/2$.
Let $\Xm \sim \mu^{\otimes m}$ and $\Ym \sim \nu^{\otimes m}$, where $\mu$ and $\nu$ are two probability measures on $M$.
For any $p \geq 1$ we have
\begin{equation*}
\left\Vert \mathbb{E}_{\Psi_\mu^m}[\lambda_\Xm] - \mathbb{E}_{\Psi_\nu^m}[\lambda_\Ym]
\right\Vert_\infty \leq  2 \, m^{\frac{1}{p}} W_{p}(\mu,\nu),
\end{equation*}
where $W_{p}$ is the $p$th Wasserstein distance on $M$.
%where $\Vert f \Vert_\infty = \sup_{t \in [0,T]} |f(t)|$ is the supremum norm.
\end{theorem}

The result of Theorem \ref{thm:stabW} is useful for two reasons. First, it tells
us that for a fixed $m$, the expected "topological behavior" of a set of $m$
points carries some stable information about the underlying measure from which
the data are generated. Second, it provides a lower bound for the Wasserstein
distance between two measures, based on the topological signature of samples of
$m$ points.

\subsubsection{Asymptotic normality}

As in the previous section, we consider several persistence diagrams (or other representations). The next step after giving central tendency descriptors of  persistence homology is to provide asymptotic normality results for these quantities together with bootstrap procedures to derive confidence regions.  It is of course easier to show such results for functional representations of persistence. In \cite{chazal2013bootstrap,chazal2015stochastic}, following this strategy  confidence bands for landscapes are proposed from the observation of  landscapes $\lscape_1,\dots,\lscape_N$ drawn i.i.d. from a random distribution in the space of landscapes. The  asymptotic validity as well as the uniform convergence of the multiplier bootstrap  is shown in this framework. Note that similar results  can be also proposed for many representations of persistence, in particular by showing that the corresponding functional spaces are Donsker spaces.

\subsection{Other statistical approaches to TDA}
Statistical approaches for \tda\ are knowing an increasing interest and many others have been proposed in the recent years or are still subject to active research activities, as illustrated in the following non-exhaustive list of examples. 

\paragraph{Hypothesis testing.} Several methods have been proposed for hypothesis testing procedures for persistent homology, mostly based on permutation strategies and for two sample testing. \cite{robinson2017hypothesis} focuses on pairwise distances of persistence diagrams whereas \cite{berry2020functional} study more general functional summaries.  Hypothesis tests based on kernel approaches have been proposed in \cite{kusano2019expectation}. A two-stage hypothesis test of filtering and testing for persistent images is also presented in\cite{moon2020hypothesis}. 

\paragraph{Persistence Homology Transform.} 
The representations introduced before  are all transformations derived from the persistence diagram computed from a fixed filtration built over a data set. 
The Persistence Homology Transform introduced in \cite{turner2014persistent,curry2018many} to study shapes in $\R^d$, takes a different path by looking at the persistence homology of the sublevel set filtration induced by the projection of the considered shape to each direction in $\R^d$.  It comes with several  interesting properties, in particular  the Persistence Homology Transform is  a sufficient statistic for distributions defined on the set of  geometric and finite  simplicial complexes embedded in $\R^d$.  

\paragraph{Bayesian statistics for TDA.} A Bayesian approach to persistence diagram inference has been proposed in \cite{maroulas2020bayesian} by viewing a persistence diagram  as a sample from a point process.  This Bayesian method   computes the point process posterior intensity based on a Gaussian mixture intensity for the prior.

%%%%%%%%%%%%%%%%%%%%%%%%%%%%%%%%%%%%%%%%%%%%%%%%%%%%
\subsection{Persistent homology and machine learning}
\label{sec:PHandLearning}
%%%%%%%%%%%%%%%%%%%%%%%%%%%%%%%%%%%%%%%%%%%%%%%%%%%%
%In some domains persistence diagrams obtained from data can be directly interpreted and exploited for better understanding of the phenomena from which the data have been generated. This, for example, the case in the study of force fields in granular media \cite{kramar2013persistence} or of atomic structures in glass \cite{0957-4484-26-30-304001} in material science, in the study of the evolution of convection patterns in fluid dynamics \cite{kramar2016analysis} or in the analysis of nanoporous structures in chemistry \cite{lee2017quantifying} where topological features can be rather clearly related to specific geometric structures and patterns in the considered data. 

%There are many other cases where persistence features cannot be easily or directly interpreted but present valuable information for further processing. However, the highly non linear nature of diagrams prevents them to be immediately used as standard features in machine learning algorithms.
%%
Using \tda\ and more specifically persistent homology for machine learning is a subject that attracts a lot of information and generated an intense research activity. Although the recent progress in this area goes far beyond the scope of this paper, we briefly introduce the main research directions with a few reference to help the newcomer to the field to get started.  

\paragraph{\tda\ for exploratory data analysis and descriptive statistics.}
In some domains, \tda\ can be fruitfully used as a tool for exploratory analysis and visualization. For example, the Mapper algorithm provides a powerful approach to explore and vizualize the global topological structure of complex data sets. In some cases, 
persistence diagrams obtained from data can be directly interpreted and exploited for better understanding of the phenomena from which the data have been generated. This is, for example, the case in the study of force fields in granular media \citep{kramar2013persistence} or of atomic structures in glass \citep{0957-4484-26-30-304001} in material science, in the study of the evolution of convection patterns in fluid dynamics \citep{kramar2016analysis}, in machining monitoring \citep{khasawneh2016chatter} or in the analysis of nanoporous structures in chemistry \citep{lee2017quantifying} where topological features can be rather clearly related to specific geometric structures and patterns in the considered data.

\paragraph{Persistent homology for feature engineering.}
There are many other cases where persistence features cannot be easily or directly interpreted but present valuable information for further processing. However, the highly non linear nature of diagrams prevents them to be immediately used as standard features in machine learning algorithms.

Persistence landscapes and linear representations of persistence diagrams offer a first option to convert persistence diagrams into elements of a vector space that can be directly used as features in classical machine learning pipelines. 
This approach has been used, for example, for protein binding \citep{kovacev2016using}, object recognition \citep{loc-pbsr-14} or time-series analysis. In the same vein, the construction of kernels for persistence diagrams that preserve their stability properties has recently attracted some attention. Most of them have been obtained by considering diagrams as discrete measures in $\R^2$. Convolving a symmetrized (with respect to the diagonal) version of persistence diagrams with a 2D Gaussian distribution, \cite{reininghaus2015stable} introduce a multi-scale kernel and apply it to shape classification and texture recognition problems. Considering Wasserstein distance between projections of persistence diagrams on lines,  \cite{CarriereOudot2017ICML} build another kernel and test its performance on several benchmarks. Other kernels, still obtained by considering persistence diagrams as measures, have also been proposed in \cite{kusano2017kernel}.

Various other vector summaries of persistence diagrams have been proposed and then used as features for different problems. For example, basic summaries are considered in \cite{DBLP:conf/ctic/BonisOOC16} and combined with quantization and pooling methods to address non rigid shape analysis problems;  Betti curves  extracted from persistence diagrams are used with 1-dimensional Convolutional Neural Networks (CNN) to analyze time dependent data and recognize human activities from inertial sensors in \cite{umeda2017time,dindin2020topological}; \emph{persistence images} are introduced in \cite{adams2017persistence} and are considered to address some inverse problems using linear machine learning models in \cite{obayashi2017persistence}.

The above mentioned kernels and vector summaries of persistence diagrams are built independently of the considered data analysis or learning task. Moreover, it appears that in many cases the relevant topological information is not carried by the whole persistence diagrams but is concentrated in some localized regions that may not be obvious to identify. This usually makes the choice of a relevant kernel or vector summary very difficult for the user. To overcome this issue, various authors have proposed 
learning approaches that allows to learn the relevant topological features for a  given task. In this direction, \cite{hofer2017deep} proposes a deep learning approach to learn parameters of persistence images representations of persistence diagrams while \cite{kim2020pllay} introduce a neural network layer for persistence landscapes. In \cite{carriere2020perslay}, the authors introduce a general neural network layer for persistence diagrams that can be either used to learn an appropriate vectorization or directly integrated in a deep neural network architecture. Other methods, inspired from k-means, propose unsupervised methods to vectorize persistence diagrams \cite{royer2019atol,zielinski2018persistence}, some of them coming with theoretical guarantees \cite{chazal2020optimal}.

\paragraph{Persistent homology for machine learning architecture optimization and model selection.}
More recently, \tda\ has found new developments in machine learning where persistent homology is no longer used for feature engineering but as a tool to design, improve or select models - see, e.g. \cite{rieck2019neural, moor2020topological,pmlr-v97-hofer19a, chen2019topological, carlsson2020topological, gabrielsson2019exposition,ramamurthy2019topological}. Many of these tools rely on the introduction of loss or regularization functions depending on persistent homology features, raising the problem of their optimization. Building on the powerful tools provided by software libraries such as PyTorch or TensorFlow, practical methods allowing to encode and optimize a large family of persistence-based functions have been proposed and experimented \cite{bruel2019topology, poulenard2018topological}. A general framework for persistence-based function optimization based on stochastic subgradient descent algorithms with convergence guarantees has been recently proposed and implemented in a easy-to-use software tool \cite{carriere2020note}. With a different perspective, another theoretical framework to study the differentiable structure of functions of persistence diagrams has been proposed in \cite{leygonie2019framework}.

%
%\fred{
%\begin{itemize}
%\item persistence and time-delay embedding for time series analysis: \cite{seversky2016time}, \cite{tralie2017quasi} and also \cite{umeda2017time}... MAis est-ce le bon endroit pour en parler?
%\end{itemize}
%}

\section{\tda\ for data sciences with the GUDHI library} \label{sec:gudhi}

In this section we illustrate TDA methods with the  {Python library Gudhi\footnote{\url{http://gudhi.gforge.inria.fr/python/latest/}} }~\citep{maria2014gudhi} together with popular libraries as numpy \citep{walt2011numpy}, scikit-learn \citep{pedregosa2011scikit}, pandas \citep{mckinney2010data}. More illustrations with python notebooks can be found in the Tutorial Github page\footnote{https://github.com/GUDHI/TDA-tutorial} of Gudhi.

\subsection{Bootstrap and comparison of protein binding configurations}

This example is borrowed from \cite{kovacev2016using}. In this paper, persistent homology is used to analyze protein binding and more precisely it compares closed  and  open  forms  of  the  maltose-binding  protein (MBP), a  large  biomolecule  consisting  of  370  amino  acid  residues. The analysis is not based on geometric distances in $\mathbb R^3$ but on a metric of {\it dynamical distances}  defined by
$$D_{ij} = 1 - |C_{ij}|,$$
where C is the correlation matrices between residues.
The data can be download at this link\footnote{ \url{https://www.researchgate.net/publication/301543862_corr}}.
%, we thank  the authors for making the data available.
{\small
%, caption=Block of statements for computing and plotting persistence diagram for one configuration of MBP.]
\begin{lstlisting}[language=Python]
import numpy as np
import gudhi as gd
import pandas as pd
import seaborn as sns

corr_protein = pd.read_csv("mypath/1anf.corr_1.txt", 
                           header=None,
                           delim_whitespace=True)
dist_protein_1 = 1- np.abs(corr_protein_1.values)
rips_complex_1= gd.RipsComplex(distance_matrix=dist_protein_1,
                                  max_edge_length=1.1)
simplex_tree_1 = rips_complex_1.create_simplex_tree(max_dimension=2)
diag_1 = simplex_tree_1.persistence()
gd.plot_persistence_diagram(diag_1)
\end{lstlisting}
}

\begin{figure}[h!] 
\centering 
\includegraphics[width = 1 \columnwidth]{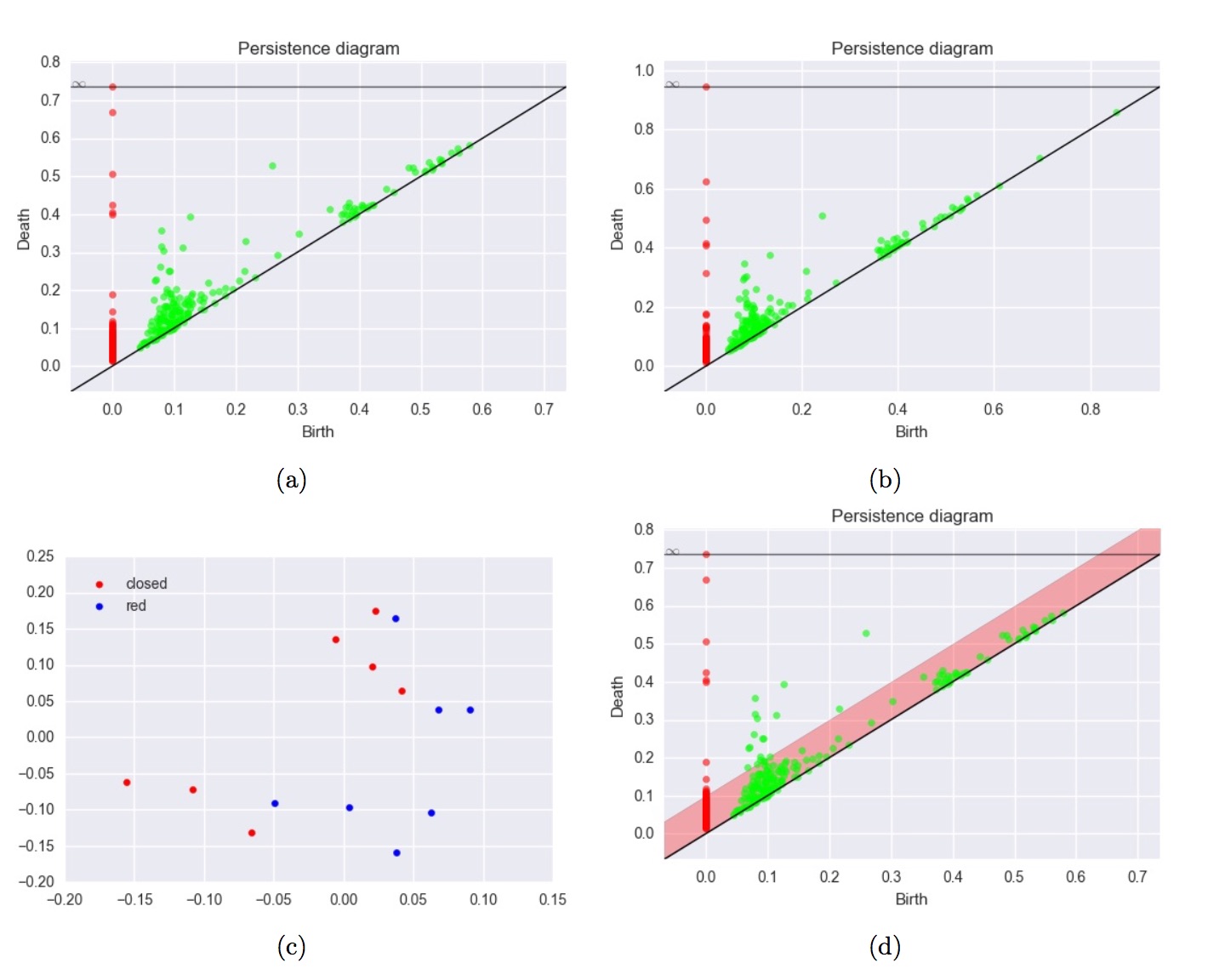}
\caption{(a) and (b) two persistence diagrams for  two configurations of MBP. (c) MDS configuration  for the matrix of bottleneck distances. (d) Persistence diagram and confidence region for the persistence diagram of  a MBP.}
\label{fig:Protein4graphs} 
\end{figure}

%\begin{figure}[h!] 
%\centering 
%\includegraphics[width = 0.49 \columnwidth]{UneProtein_1_0.jpg}
%\includegraphics[width = 0.49 \columnwidth]{UneProtein_2_0.jpg}
%\caption{Persistence diagrams for  two configurations of MBP.}
%\label{fig:TwoPersDiag} 
%\end{figure}
\noindent For comparing persistence diagrams, we use the bottleneck distance. The block  of statements given below  computes persistence intervals and computes the bottleneck distance for 0-homology and 1-homology:
{\small
\begin{lstlisting}[language=Python]
interv0_1 = simplex_tree_1.persistence_intervals_in_dimension(0)
interv0_2 = simplex_tree_2.persistence_intervals_in_dimension(0)
bot0  = gd.bottleneck_distance(interv0_1,interv0_2)

interv1_1 = simplex_tree_1.persistence_intervals_in_dimension(1)
interv1_2 = simplex_tree_2.persistence_intervals_in_dimension(1)
bot1  = gd.bottleneck_distance(interv1_1,interv1_2)
\end{lstlisting}
}
\noindent In this way, we can compute the matrix of bottleneck distances  between the fourteen  MPB. Finally, we apply a multidimensional scaling method to find a configuration in $\mathbb R^2$ which almost match with the bottleneck distances, see Figure~\ref{fig:Protein4graphs}(c). We use the scikit-learn library for the MDS:
{\small
%, caption= Multidimensional scaling on the matrix of bottleneck distances.]
\begin{lstlisting}[language=Python]
import matplotlib.pyplot as plt
from sklearn import manifold

mds = manifold.MDS(n_components=2, dissimilarity="precomputed")
config = mds.fit(M).embedding_

plt.scatter(config [0:7,0], config [0:7, 1], color='red', label="closed")
plt.scatter(config [7:l,0], config [7:l, 1], color='blue', label="red")
plt.legend(loc=1)
\end{lstlisting}
}
%\begin{figure}[h!]
%\centering
%\includegraphics[width = 0.49 \columnwidth]{MDS.jpg}
%\caption{MDS configuration  for the matrix of bottleneck distances.}
% \label{fig:MDS}
%\end{figure}

\medskip

We now define a    confidence band for a diagram using the bottleneck bootstrap approach. We resample over the lines (and columns) of the matrix of distances and we compute the bottleneck distance between the original persistence diagram and the bootstrapped persistence diagram. We repeat the procedure many times and  finally we estimate the quantile 95$\%$ of this collection of bottleneck distances. We take the value of the quantile to define a confidence band on the original diagram (see Figure~\ref{fig:Protein4graphs}(d)). However, such a procedure should be considered with caution because as far as we know the validity of the bottleneck bootstrap has not been proved in this framework.
%\begin{figure}[h!] 
%\centering 
%\includegraphics[width = 0.49 \columnwidth]{BandProtein.jpg}
%\caption{Persistence diagram and confidence region for the persistence diagram of  a MBP.}
%\label{fig:OneProtein}
%\end{figure}

\subsection{Classification for sensor data}

In this experiment, the 3d acceleration of 3 walkers (A, B and C) have been recorded from the sensor of a smart phone\footnote{The dataset can be download at this link \url{http://bertrand.michel.perso.math.cnrs.fr/Enseignements/TDA/data_acc}}. Persistence homology is not sensitive to the choice of axes and so no preprocessing is  necessary to align the 3 times series according to the same axis. From these three times series, we have picked at random sequences of 8 seconds in the complete time series,  that is 200 consecutive points of acceleration in $\R^3$. For each walker, we extract 100 time series in this way. 
The next block of statements computes the persistence for the alpha complex filtration for \verb+data_A_sample+, one of the 100 times series of acceleration of Walker A.
{\small
%, caption= Block of statements of computing the persistence homology from the alpha shape filtration of one time series of acceleration of Walker A.]
\begin{lstlisting}[language=Python]
alpha_complex_sample = gd.AlphaComplex(points = data_A_sample)
simplex_tree_sample = alpha_complex_sample.create_simplex_tree(max_alpha_square=0.3) 
diag_Alpha = simplex_tree_sample.persistence()
\end{lstlisting}
}
\noindent From \verb+diag_Alpha + we can then easily compute and plot the persistence landscapes, see Figure~\ref{fig:landscForest}(a). For all the 300 times series, we compute the persistence landscapes for dimension 0 and 1 and we compute the three first landscapes for the 2 dimensions. Moreover, each persistence landscape is discretized on 1000 points. Each time series is thus described by 6000 topological variables. To predict the walker from these features, we use a random forest \citep{breiman2001random}, which is known to be an efficient in such an high dimensional setting. We split the data into train and test samples at random several times. We finally obtain an averaged classification error around 0.95. We can also visualize the most important variables in the Random Forest, see Figure~\ref{fig:landscForest}(b).

\begin{figure}[h] 
\centering 
\includegraphics[width = 1 \columnwidth]{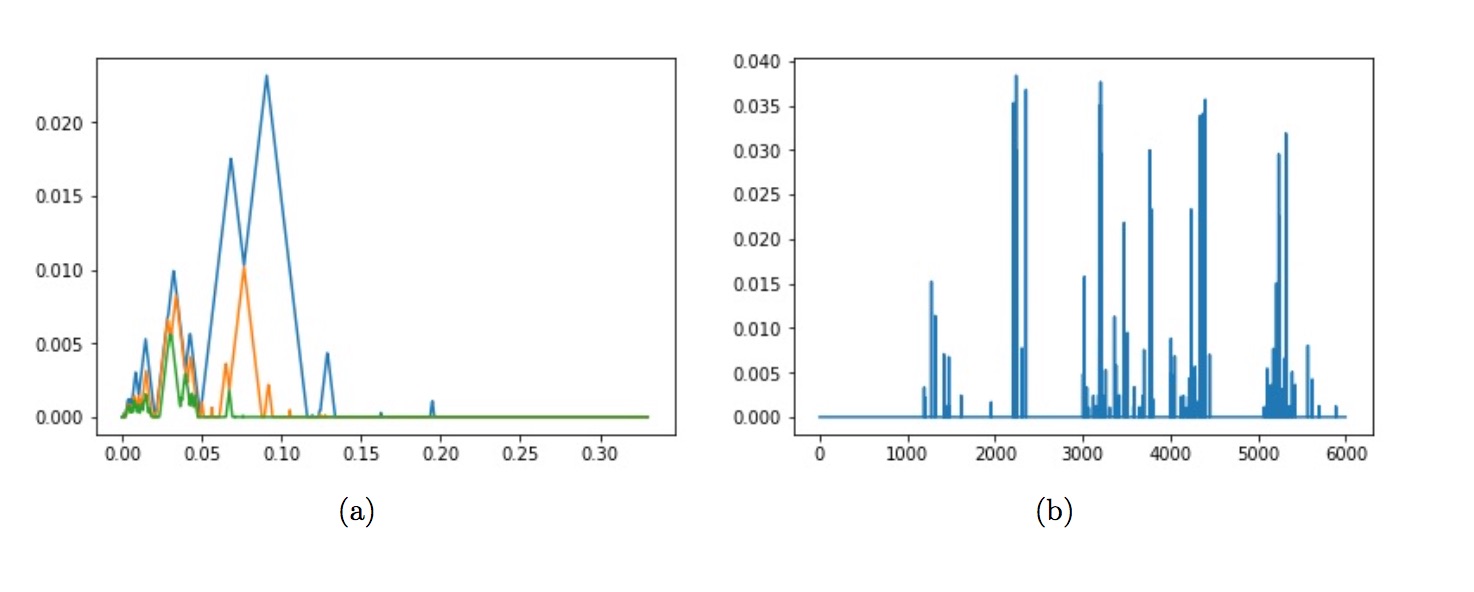}
\caption{(a)The three first landscapes for 0-homology of the alpha shape filtration defined for a time series of acceleration of Walker A. (b) Variable importances of the landscapes coefficients for the classification of Walkers. The 3 000 first coefficients correspond to the three landscapes of dimension 0 and the 3 000 last coefficients to the three landscapes of dimension 1. There are 1000 coefficients per landscape. Note that the first landscape of dimension 0 is always the same using the Rips complex (a trivial landscape)  and consequently the corresponding coefficients have a zero importance value.}
\label{fig:landscForest}
\end{figure}

%\section{Discussion}
%
%\textcolor{red}{directions de recherche}
%
%
%One important direction of research in the field concern statistical learning with persistent homology. The space of persistence diagrams is a metric spaces and standard statistical learning methods require a Hilbert structure on the descriptors space. Nearest neighbors strategies can of course be applied in such this context. Another approach consists in applying standard statistical learning methods to landscapes or other topological signatures living in an Hilbert space \cite{bubenik2015statistical,chazal2014subsampling}. An important direction of research last years is about providing kernels for persistence diagrams, for instance by defining  implicit feature maps and then computing directly kernels on the space of persistence diagrams~\cite{reininghaus2015stable,CarriereOudot2017ICML}.
%

\paragraph{Acknowledgements}
%\begin{acknowledgement}
%\fred{We should update this part: TopAI + ...}
This work was partly supported by the French ANR chair in Artificial Intelligence TopAI. 
We thank the authors of \cite{kovacev2016using} for making their data available.
%\end{acknowledgement}

\bibliographystyle{apalike}
\bibliography{refTDA}
\end{document}